\documentclass[12pt, reqno]{amsart}
\usepackage{amssymb,latexsym,amsmath,amscd,amsthm,graphicx, color}
\usepackage[all]{xy}
\usepackage{pgf,tikz}
\usepackage{mathrsfs}
\usetikzlibrary{arrows}
\usepackage[left=0.6 in, top=0.6 in, right=0.6 in, bottom=0.3 in]{geometry}
\makeatletter
\newcommand{\setword}[2]{%
  \phantomsection
  #1\def\@currentlabel{\unexpanded{#1}}\label{#2}%
}
\makeatother

\usepackage[linkcolor=blue, urlcolor=blue, citecolor=blue,
colorlinks, bookmarks]{hyperref}

\definecolor{uuuuuu}{rgb}{0.26666666666666666,0.26666666666666666,0.26666666666666666}
\definecolor{xdxdff}{rgb}{0.49019607843137253,0.49019607843137253,1.}
\definecolor{ffqqqq}{rgb}{1.,0.,0.}
\definecolor{ffqqqq}{rgb}{1.,0.,0.}
\definecolor{ffxfqq}{rgb}{1.,0.4980392156862745,0.}

\definecolor{uuuuuu}{rgb}{0.26666666666666666,0.26666666666666666,0.26666666666666666}
\definecolor{qqwuqq}{rgb}{0.,0.39215686274509803,0.}
\definecolor{zzttqq}{rgb}{0.6,0.2,0.}
\definecolor{xdxdff}{rgb}{0.49019607843137253,0.49019607843137253,1.}
\definecolor{qqqqff}{rgb}{0.,0.,1.}
\definecolor{cqcqcq}{rgb}{0.7529411764705882,0.7529411764705882,0.7529411764705882}
\definecolor{sqsqsq}{rgb}{0.12549019607843137,0.12549019607843137,0.12549019607843137}

\theoremstyle{plain}

\newtheorem{theorem}[subsection]{Theorem}
\newtheorem{theo}[subsection]{Theorem}

\newtheorem{lemma}[subsection]{Lemma}
\newtheorem{defi}[subsection]{Definition}
\newtheorem{prop}[subsection]{Proposition}

\theoremstyle{definition}

\newtheorem{cor}[subsection]{Corollary}
\newtheorem{exam}[subsection]{Example}

\newtheorem{remark}[subsection]{Remark}

\newtheorem{note}[subsection]{Note}

\newtheorem*{deli}{Delineation}

\newcommand{\uu}{\cup}
\newcommand{\ii}{\cap}
\newcommand{\UU}{\bigcup}

\newcommand{\ci}{\subseteq}
\newcommand{\sci}{\subset}

\newcommand{\set}[1]{\{#1\}}

\newcommand{\ga}{\alpha}
\newcommand{\gb}{\beta}

\renewcommand{\gg}{\gamma}

\newcommand{\gk}{\kappa}

\newcommand{\gq}{\theta}

\newcommand{\tit}{\textit}

\newcommand{\D}[1]{\mathbb{#1}}

\newcommand{\te}{\text}

\newcommand{\ol}{\overline}
\newcommand{\ul}{\underline}

\newcommand{\tri}{\triangle}
\begin{document}
\title{Conditional constrained and unconstrained quantization for probability distributions}

\author{$^1$Megha Pandey}
 \author{$^2$Mrinal Kanti Roychowdhury}

\address{$^{1}$Department of Mathematical Sciences \\
Indian Institute of Technology (Banaras Hindu University)\\
Varanasi, 221005, India.}
\address{$^{2}$School of Mathematical and Statistical Sciences\\
University of Texas Rio Grande Valley\\
1201 West University Drive\\
Edinburg, TX 78539-2999, USA.}

\email{$^1$meghapandey1071996@gmail.com, $^2$mrinal.roychowdhury@utrgv.edu}


\date{}
\maketitle
\pagestyle{myheadings}\markboth{M. Pandey and M.K. Roychowdhury}{Conditional constrained and unconstrained quantization for probability distributions}
\begin{abstract}
In this paper, we introduce and develop the concept of \emph{conditional quantization} for Borel probability measures on \( \mathbb{R}^k \), considering both \emph{constrained} and \emph{unconstrained} frameworks. For each setting, we define the associated quantization errors, dimensions, and coefficients, and provide explicit computations for specific classes of probability distributions. A key result in the unconstrained case is that the union of all optimal sets of \( n \)-means is dense in the support of the measure. Furthermore, we demonstrate that in conditional constrained quantization, if the conditional set is contained within the union of the constraint family, then the lower and upper quantization dimensions, as well as the corresponding coefficients, remain unaffected by the conditional set for any Borel probability measure. In contrast, if the conditional set is not contained within this union, these properties may no longer hold, as illustrated through various examples.
\end{abstract}

\section{Introduction}
Quantization for a Borel probability measure refers to the idea of estimating a given probability by a discrete probability supported on a finite set. 
A plethora of research is given on quantization for probability distributions without using any constraint. The concept of constrained quantization was recently introduced by Pandey and Roychowdhury (see \cite{PR1, PR2, PR3}). This new approach allows us to categorize quantization into two types: unconstrained quantization and constrained quantization. Unconstrained quantization is traditionally known as quantization. For some recent work in the direction of unconstrained quantization, one can see \cite{GL, DFG, DR, GL2, GL3, KNZ, PRRSS, P1, R1, R2, R3}. Quantization theory has broad applications in communications, information theory, signal processing, and data compression (see \cite{GG, GL1, GN, P, Z1, Z2}). This paper deals with conditional quantization in both constrained and unconstrained scenarios. Conditional quantization also has significant interdisciplinary applications: for example, in radiation therapy of cancer treatment to find the optimal locations of $n$ centers of radiation, where $k$ centers for some $k<n$ of radiation are preselected, the conditional quantization technique can be useful.

Let $P$ be a Borel probability measure on $\D R^k$ equipped with a metric $d$ induced by a norm $\|\cdot\|$ on $\D R^k$, and $r \in (0, \infty)$. Let $\D N:=\set{1, 2, 3, \cdots}$ be the set of natural numbers. For a finite set $\gg \sci \D R^k$ and $a\in \gg$, by $M(a|\gg)$ we denote the set of all elements in $\D R^k$ which are nearest to $a$ among all the elements in $\gg$, i.e.,
$M(a|\gg)=\set{x \in \D R^k : d(x, a)=\mathop{\min}\limits_{b \in \gg}d(x, b)}.$
$M(a|\gg)$ is called the \tit{Voronoi region} in $\D R^k$ generated by $a\in \gg$.
\begin{defi}\label{EqVr111}
Let $\set{S_j\ci \D R^k: j\in \D N}$ be a family of closed sets with $S_1$ nonempty. Let $\gb\sci \D R^k$ be given with $\te{card}(\gb)=\ell$ for some $\ell\in \D N$. 
Then, for $n\in \D N$ with $n\geq \ell$, the \tit {$n$th conditional constrained quantization
error} for $P$, of order $r$, with respect to the family of constraints $\set{S_j\ci \D R^k: j\in \D N}$ and the set $\gb$, is defined by
\begin{equation}\label{EqVr121} 
V_{n, r}:=V_{n, r}(P)=\inf_{\ga} \Big\{\int \mathop{\min}\limits_{a\in\ga\uu\gb} d(x, a)^r dP(x) : \ga \ci \UU_{j=1}^n S_j, ~ 0\leq  \text{card}(\ga) \leq n-\ell \Big\},
\end{equation} 
where $\te{card}(A)$ represents the cardinality of the set $A$. 
\end{defi}

Notice that for a given $r\in (0, \infty)$, the error $V_{n, r}$ explicitly depends on $\gb$. 
For any $\ga\ci \D R^k$, where $\ga$ is locally finite (i.e., intersection of $\ga$ with any bounded subset of $\D R^k$ is finite, in other words,  $\ga$ is countable and closed),
the number 
\begin{equation}
  V_r(P; \ga):= \int \mathop{\min}\limits_{a\in\ga} d(x, a)^r dP(x)
\end{equation}
is called the \tit{distortion error} for $P$, of order $r$, with respect to the set $\ga$. 
We assume that $\int d(x, 0)^r dP(x)<\infty$ to make sure that the infimum in \eqref{EqVr121} exists (see \cite{PR1}). The set $\gb$ that occurs in Definition~\ref{EqVr111} is called a \tit{conditional set}. 
\begin{defi}
A set $ \ga\uu\gb$, where $\ga \ci \mathop{\UU}_{j=1}^n S_j$ and $P(M(b|\ga\uu \gb))>0$ for $b\in \gb$, for which the infimum in  \eqref{EqVr121} exists and contains no less than $\ell$ elements, and no more than $n$ elements is called an \tit{optimal set of $n$-points} for $P$ or more specifically a \tit{conditional optimal set of $n$-points} for $P$ with respect to the family of constraints $\set{S_j\ci \D R^k: j\in \D N}$ and the conditional set $\gb$. Elements of an optimal set are called \tit{optimal elements}. 
\end{defi}

\begin{remark} \label{EqVr2}
Instead of the family of constraints $\set{S_j \ci \D R^k : j\in \D N}$ if there is a single constraint $S$, i.e.,  if $S_j=S$ for all $j\in \D N$, then Definition~\ref{EqVr111} reduces to 
\[
V_{n, r}:=V_{n, r}(P)=\inf_{\ga} \Big\{\int \mathop{\min}\limits_{a\in\ga\uu\gb} d(x, a)^r dP(x) : \ga \ci S, ~ 0\leq  \text{card}(\ga) \leq n-\ell \Big\},
\]
which is called the \tit {$n$th conditional constrained quantization
error} for $P$, of order $r$, with respect to the constraint $S$ and the conditional set $\gb$.
\end{remark}
 Let $V_{n, r}(P)$ be a strictly decreasing sequence, and write $V_{\infty, r}(P):=\mathop{\lim}\limits_{n\to \infty} V_{n, r}(P)$.   The numbers
\begin{equation} \label{eq55} \ul D_r(P):=\liminf_{n\to \infty}  \frac{r\log n}{-\log (V_{n, r}(P)-V_{\infty, r}(P))} \te{ and } \ol D_r(P):=\limsup_{n\to \infty} \frac{r\log n}{-\log (V_{n, r}(P)-V_{\infty, r}(P))}, \end{equation}
are called the \tit{conditional lower} and the \tit{conditional upper constrained quantization dimensions} of the probability measure $P$ of order $r$, respectively. If $\ul D_r (P)=\ol D_r (P)$, the common value is called the \tit{conditional constrained quantization dimension} of $P$ of order $r$ and is denoted by $D_r(P)$. The conditional constrained quantization dimension measures the speed how fast the specified measure of the conditional constrained quantization error converges as $n$ tends to infinity. A higher conditional constrained quantization dimension suggests a faster convergence of the $n$th conditional constrained quantization error.

 For any $\gk>0$, the two numbers $\liminf_n n^{\frac r \gk}  (V_{n, r}(P)-V_{\infty, r}(P))$ and $\limsup_n  n^{\frac r \gk}(V_{n, r}(P)-V_{\infty, r}(P))$ are, respectively, called the \tit{$\gk$-dimensional conditional lower} and \tit{conditional upper constrained quantization coefficients} for $P$ of order $r$. If both of them are equal, then it is called the \tit{$\gk$-dimensional conditional constrained quantization coefficient} for $P$ of order $r$.  

\begin{defi} 
In Definition~\ref{EqVr111} if $S_j=\D R^k$ for all $j\in \D N$, then  for $n\in \mathbb{N}$ with $n\geq \ell$, the \tit {$n$th conditional unconstrained quantization
error} for $P$, of order $r$, with respect to the conditional set $\gb$, is defined by 
\begin{equation}\label{EqVr11} 
V_{n, r}:=V_{n, r}(P)=\inf_{\ga} \Big\{\int \mathop{\min}\limits_{a\in\ga\uu\gb} d(x, a)^r dP(x) : \ga \ci \D R^k, ~ 0\leq  \text{card}(\ga) \leq n-\ell \Big\}.
\end{equation} 
The corresponding quantization dimension and the $\gk$-dimensional quantization coefficient, if they exist, are called the conditional unconstrained quantization dimension and the $\gk$-dimensional conditional unconstrained quantization coefficient for $P$, respectively. A set $\ga\uu \gb$ for which the infimum in \eqref{EqVr11} exists is called a conditional unconstrained optimal set of $n$-points for $P$. 
\end{defi} 

This paper deals with $r=2$ and $k=2$, and the metric on $\D R^2$ as the Euclidean metric induced by the Euclidean norm $\|\cdot\|$. Instead of writing $V_r(P; \ga)$ and $V_{n, r}:=V_{n, r}(P)$ we will write them as $V(P; \ga)$ and $V_n:=V_{n}(P)$, i.e., $r$ is omitted in the subscript as $r=2$ throughout the paper.

\begin{deli}
In this paper,  first we have given the preliminaries. Then, in Section~\ref{secMe1}, we have investigated the conditional constrained quantization for a uniform distribution on the boundary of a semicircular disc; in Section~\ref{secMe4}, we have investigated the conditional unconstrained quantization for a uniform distribution on an equilateral triangle; in Section~\ref{secob}, with some examples, we have described some properties of conditional optimal sets of $n$-points. In the last section, Section~\ref{prop}, we have proved two theorems: first theorem shows that if $\ga_n$ are the optimal sets of $n$-means for $P$ for all $n\in \D N$, then $\mathop{\uu}\limits_{n=1}^\infty \ga_n$ is dense in the support of $P$; in the second theorem we have shown the facts that in conditional constrained quantization, the lower and upper quantization dimensions, and the lower and upper quantization coefficients for a Borel probability measure do not depend on the conditional set provided that the conditional set is contained in the union of the family of constraints. If the conditional set is not contained in the union of the family of constraints, the facts may not be true as illustrated through some examples.  
\end{deli}

\begin{note}
The conditional quantization in both constrained and unconstrained scenario is newly introduced in this paper. Although all the work done in the sequel deal with uniform distributions, interested researchers can explore them for any probability distribution. 
 \end{note}

\section{Preliminaries} \label{secPre}

Let $\mathbb{N}$ be the set of natural numbers and $\mathbb{R}$ be the collection of all real numbers.
For any two elements $(a, b)$ and $(c, d)$ in $\D R^2$, we write 
 \[\rho((a, b), (c, d)):=(a-c)^2 +(b-d)^2,\] which gives the squared Euclidean distance between the two elements $(a, b)$ and $(c, d)$.
  Two elements $p$ and $q$ in an optimal set of $n$-points are called \tit{adjacent elements} if they have a common boundary in their own Voronoi regions. Let $e$ be an element on the common boundary of the Voronoi regions of two adjacent elements $p$ and $q$ in an optimal set of $n$-points. Since the common boundary of the Voronoi regions of any two elements is the perpendicular bisector of the line segment joining the elements, we have
\[\rho(p, e)-\rho(q, e)=0. \]
We call such an equation a \tit{canonical equation}. 
Notice that any element $x\in \D R$ can be identified as an element $(x, 0)\in \D R^2$. Thus, the nonnegative real-valued function $\rho$ on $\D R \times \D R^2$ defined by 
\[\rho: \D R \times \D R^2 \to [0, \infty) \te{ such that } \rho(x, (a, b))=(x-a)^2 +b^2,\]
represents the squared Euclidean distance between an element $x\in \D R$ and an element $(a, b)\in \D R^2$. On the other hand, 
\[\rho: \D R \times \D R  \to [0, \infty) \te{ such that } \rho(x, y)=x^2+y^2,\]
represents the squared Euclidean distance between any two elements $x, y \in \D R$.

Let us now give the following three propositions. 
\begin{prop} \label{Me0} 
Let $P$ be a uniform distribution on the closed interval $[a, b]$. Let $a\leq c<d\leq b$. Let $\ga_n$ be an optimal set of $n$-points for $P$ such that $\ga_n$ contains $m$ elements from the closed interval $[c, d]$ including the endpoints $c$ and $d$, then
\[\ga_n\ii[c, d]=\set{c+\frac{j-1}{m-1}(d-c) : 1\leq j\leq m}.\] 
Then, the distortion error contributed by these $m$ elements in the closed interval  $[c, d]$ is given by 
 \[V(P, \ga_n\ii [c, d]):=\frac 1 {12}\frac{(d-c)^3}{b-a}\frac 1{ (m-1)^2}.\]
\end{prop} 
  
  \begin{proof}
$P$ being a uniform distribution on the closed interval $[a, b]$, its density function is given by $f(x)=\frac {1}{b-a}$ if $x\in [a, b]$, and zero otherwise. Also, notice that 
the closed interval $[a, b]$ can be represented by 
\[[a, b]:=\set{t : a\leq t\leq b}.\]
Let ${c_1, c_2, c_3, \cdots, c_m}$ be the $m$ elements that $\ga_n$ contains from the closed interval $[c, d]$ such that $c=c_1<c_2<\cdots<c_m=d$. 
Since the closed interval $[c, d]$ is a line segment and $P$ is a uniform distribution, we have 
\[c_2-c_1=c_3-c_2=\cdots=c_{m}-c_{m-1}=\frac{c_m-c_1}{m-1}=\frac {d-c}{m-1}\]
implying 
\begin{align*}
c_2&=c_1+\frac {d-c}{m-1}=c+\frac {d-c}{m-1},\\
c_3&=c_2+\frac {d-c}{m-1}=c+\frac {2(d-c)}{m-1}, \\
c_4&=c_3+\frac {d-c}{m-1}=c+\frac {3(d-c)}{m-1},\\
&\te{and so on.}
\end{align*} 
Thus, we have $c_j= c+\frac{j-1}{m-1}(d-c)$ for $ 1\leq j\leq m$. The distortion error contributed by the $m$ elements in the closed interval  $[c, d]$  is given by
 \begin{align*}
V(P; \ga_n\ii [c, d])&=\int_{[c,  d]} \min_{x\in \ga_n\ii[c, d]} \rho(t, x)\, dP\\
&=\frac 1{b-a} \Big(2\int_{c_1}^{\frac  {c_1+c_2}2} \rho(t, c_1)\,dt+(m-2) \int_{\frac  {c_1+c_2}2}^{\frac  {c_2+c_3}2} \rho(t, c_2)\,dt\Big)\\
&=\frac 1 {12}\frac{(d-c)^3}{b-a}\frac 1{ (m-1)^2}.
\end{align*} 
Thus, the proof of the proposition is complete. 
  \end{proof}

 \begin{prop} \label{Me030} 
Let $P$ be a uniform distribution on the closed interval $[a, b]$. Let $\ga_n$ be an optimal set of $n$-points for $P$ such that $\ga_n$ contains $n$ elements from the closed interval $[a, b]$ including the endpoint $a$. Then,
\[\ga_n=\left\{a+\frac{2(j-1)(b-a)}{2n-1} : 1\leq j\leq n\right\},\] 
with the conditional unconstrained quantization error 
 \[V_n=\frac{(b-a)^2}{3 (2 n-1)^2}.\]
\end{prop} 
 \begin{proof}
As mentioned in Proposition~\ref{Me0}, here the probability density function is given by $f(x)=\frac {1}{b-a}$ if $x\in [a, b]$, and zero otherwise, and 
\[[a, b]:=\set{t : a\leq t\leq b}.\]
Let ${c_1, c_2, c_3, \cdots, c_n}$ be the $n$ elements that $\ga_n$ contains from the closed interval $[a, b]$ including the endpoint $a$, i.e., $c_1=a$. Let $c_n=d$, where $d\in [a, b]$ is such that $d<b$. 
Since the closed interval $[a, d]$ is a line segment and $P$ is a uniform distribution, we have 
\[c_2-c_1=c_3-c_2=\cdots=c_{n}-c_{n-1}=\frac{c_n-c_1}{n-1}=\frac {d-a}{n-1}\]
implying 
\begin{align*}
c_2&=c_1+\frac {d-a}{n-1}=a+\frac {d-a}{n-1},\\
c_3&=c_2+\frac {d-a}{n-1}=a+\frac {2(d-a)}{n-1}, \\
c_4&=c_3+\frac {d-a}{n-1}=a+\frac {3(d-a)}{n-1},\\
&\te{and so on.}
\end{align*} 
Thus, we have $c_j= a+\frac{j-1}{n-1}(d-a)$ for $ 1\leq j\leq n$. The distortion error contributed by the $n$ elements is given by
 \begin{align*}
V(P; \ga_n)&=\int\min_{x\in \ga_n} \rho(t, x)\, dP\\
&=\frac 1{b-a} \Big( \int_{c_1}^{\frac  {c_1+c_2}2} \rho(t, c_1)\,dt+(n-2) \int_{\frac  {c_1+c_2}2}^{\frac  {c_2+c_3}2} \rho(t, c_2)\,dt \\
&\qquad +\int_{\frac  {c_{n-1}+c_n}2}^b \rho(t, c_n)\,dt\Big)\\
&=\frac{\frac{(a-d)^3}{(n-1)^2}-4 (b-d)^3}{12 (a-b)},
\end{align*} 
the minimum value of which is $\frac{(b-a)^2}{3 (2 n-1)^2}$ and it occurs when $d=b-\frac{b-a}{2 n-1}$. Putting the values of $d$, we have
\[c_j=a+\frac{2(j-1)(b-a)}{2n-1} \te{ for } 1\leq j\leq n\] 
with the conditional unconstrained quantization error 
 \[V_n=\frac{(b-a)^2}{3 (2 n-1)^2}.\]
Thus, the proof of the proposition is complete. 
  \end{proof}
  
   \begin{prop} \label{Me050} 
Let $P$ be a uniform distribution on the closed interval $[a, b]$. Let $\ga_n$ be an optimal set of $n$-points for $P$ such that $\ga_n$ contains $n$ elements from the closed interval $[a, b]$ including the endpoint $b$. Then
\[\ga_n=\left\{a+\frac{(2j-1)(b-a)}{2n-1} : 1\leq j\leq n\right\},\] 
with the conditional unconstrained quantization error 
 \[V_n=\frac{(b-a)^2}{3 (2 n-1)^2}.\]
\end{prop} 
 \begin{proof}
As mentioned in Proposition~\ref{Me0}, here the probability density function is given by $f(x)=\frac {1}{b-a}$ if $x\in [a, b]$, and zero otherwise, and 
\[[a, b]:=\set{t : a\leq t\leq b}.\]
Let ${c_1, c_2, c_3, \cdots, c_n}$ be the $n$ elements that $\ga_n$ contains from the closed interval $[a, b]$ including the endpoint $b$, i.e., $c_n=b$.  Let $c_1=d$, where $d\in [a, b]$ is such that $a<d$. 
Since the closed interval $[c_1, b]$ is a line segment and $P$ is a uniform distribution, we have 
\[c_2-c_1=c_3-c_2=\cdots=c_{n}-c_{n-1}=\frac{c_n-c_1}{n-1}=\frac {b-d}{n-1}\]
implying 
\begin{align*}
c_2&=c_1+\frac {b-a}{n-1}=d+\frac {b-d}{n-1},\\
c_3&=c_2+\frac {b-a}{n-1}=d+\frac {2(b-d)}{n-1}, \\
c_4&=c_3+\frac {b-a}{n-1}=d+\frac {3(b-d)}{n-1},\\
&\te{and so on.}
\end{align*} 
Thus, we have $c_j= d+\frac{j-1}{n-1}(b-d)$ for $ 1\leq j\leq n$. The distortion error contributed by the $n$ elements is given by
 \begin{align*}
V(P; \ga_n)&=\int\min_{x\in \ga_n} \rho(t, x)\, dP\\
&=\frac 1{b-a} \Big( \int_{a}^{\frac  {c_1+c_2}2} \rho(t, c_1)\,dt+(n-2) \int_{\frac  {c_1+c_2}2}^{\frac  {c_2+c_3}2} \rho(t, c_2)\,dt \\
&\qquad +\int_{\frac  {c_{n-1}+c_n}2}^{c_n} \rho(t, c_n)\,dt\Big)\\
&=\frac{4 (a-d)^3+\frac{(d-b)^3}{(n-1)^2}}{12 (a-b)},
\end{align*} 
the minimum value of which is $\frac{(b-a)^2}{3 (2 n-1)^2}$ and it occurs when $d=a+\frac{b-a}{2 n-1}$. Putting the values of $d$, we have
\[c_j=a+\frac{(2j-1)(b-a)}{2n-1} \te{ for } 1\leq j\leq n\] 
with the conditional unconstrained quantization error 
 \[V_n=\frac{(b-a)^2}{3 (2 n-1)^2}.\]
Thus, the proof of the proposition is complete. 
  \end{proof}
  
  \begin{remark}
 Although a detailed proof is given, the elements in $\ga_n$ in Proposition~\ref{Me050} can be obtained from the elements in $\ga_n$ given in Proposition~\ref{Me030} by translating the elements to the right in the amount of $\frac{b-a}{2 n-1}$.
  \end{remark} 
  
 The following theorem is known. 
\begin{theorem}(see \cite[Proposition~3.1]{PR1}) \label{sec2Theorem1}
Let $P$ be a Borel probability measure on $\D R^2$ such that $P$ is uniform on its support $\set{(x, y) \in \D R^2 : a\leq x\leq b \te{ and } y=0}$. For $n\in \D N$ with $n\geq 2$, let $\ga_n:=\{(a_i,m a_i+c): 1\leq i \leq n\}$ be an optimal set of $n$-points for $P$ so that the elements in the optimal sets lie on the line $L$ between the two elements $(d, md+c)$ and $(e, me+c)$, where $d, e\in \D R$ with $d<e$. Assume that
\[\max\set{a, (m^2+1)d+mc}=a \te{ and } \min\set{b, (m^2+1) e+mc}=b.\]
 Then, $a_i=\frac {2i-1}{2n(1+m^2)}(b-a)+\frac{a-cm}{1+m^2}$ for $1\leq i\leq n$ with constrained quantization error
\[V_2=\frac{a^2 \left(16 m^2+1\right)+2 a b \left(8 m^2-1\right)+48 a c m+b^2 \left(16 m^2+1\right)+48 b c m+48 c^2}{48 \left(m^2+1\right)} \te{ if } n=2,\]
and if $n\geq 3$, then
 \begin{align*}
V_n=&\frac 1{12 \left(m^2+1\right) n^3}\Big(-48 (a - b)^2 m^2 + (a - b) (a - b + 72 c m + 8 (11 a - 2 b) m^2) n \\
&\qquad \qquad -
 12 (a - b) m (5 c + (4 a + b) m) n^2 + 12 (c + a m)^2 n^3\Big).
\end{align*} 
 \end{theorem}

 \section{Conditional constrained quantization for a uniform distribution on the boundary of a semicircular disc} \label{secMe1}
 Let $L$ be the boundary of the semicircular disc $x_1^2+x_2^2=1$, where $x_2\geq 0$. Let the base of the semicircular disc be $AOB$, where $A$ and $B$ have the coordinates $(-1, 0)$ and $(1,0$), and $O$ is the origin $(0, 0)$. Let $s$ represent the distance of any point on $L$ from the origin tracing along the boundary $L$ in the counterclockwise direction. Notice that $L=L_1\uu L_2$,
where
\begin{align*}
L_1&=\set{(x_1, x_2) : x_1=t, \, x_2=0 \te{ for } -1\leq t\leq 1}, \te{ and }\\
L_2&=\set{(x_1, x_2) : x_1=\cos t, \, x_2=\sin t\te{ for } 0\leq t\leq \pi}.
\end{align*}
Let $P$ be the uniform distribution on the boundary of the semicircular disc. Then, the probability density function for $P$ is given by 
\begin{align*}
 f(x_1, x_2)=\left\{\begin{array}{cc}
\frac 1 {2+\pi}  & \te{ if } (x_1, x_2) \in L,\\
0 & \te{ otherwise.}
\end{array}\right.  
\end{align*}
 On both $L_1$ and $L_2$, we have $ds=\sqrt{(\frac {dx_1}{dt})^2 +(\frac {dx_2}{dt})^2} \, dt=dt$ yielding $dP(s)=P(ds)=f(x_1, x_2)ds=f(x_1, x_2) dt$.
In the definition of conditional constrained quantization error, let $\gb:=\set{(-1, 0), (1, 0)}$, and let $L$ be the constraint. Upon the given set $\gb$, the $n$th conditional constrained quantization errors are defined for all $n\geq 2$. Notice that the boundary $L$ has `maximum symmetry' with respect to the $y$-axis. By the maximum symmetry of $L$ with respect to the $y$-axis, it is meant that if two regions of the same geometrical shapes are equidistant and are on opposite sides from the line, then they have the same probability. In finding the conditional optimal sets of $n$-points, we will use this property.  

Notice that the conditional optimal set of two-points is the set $\gb$ itself. In the sequel of this section, we investigate the conditional optimal sets of $n$-points for $n\geq 3$. 
Let us now give the following proposition, which plays a vital role in this section. 

\begin{prop} \label{Me1} 
Let $\ga_n$ be a conditional optimal set of $n$-points for $P$ for some $n\geq 3$. Let $\te{card}(\ga_n\ii L_1)=n_1$  and $\te{card}(\ga_n\ii L_2)=n_2$  with corresponding conditional quantization error $V_n:=V_{n_1, n_2}(P)$ for some $n_1, n_2\geq 2$. Then, 
\[\ga_n\ii L_1=\set{(-1+\frac{2(j-1)}{n_1-1}, 0) : 1\leq j\leq n_1}, \te{ and } \ga_n\ii L_2=\set{(\cos \frac{(j-1)\pi}{n_2-1}, \sin\frac {(j-1)\pi}{n_2-1}) : 1\leq j\leq n_2},\]
with  
\[V_{n_1, n_2}(P)=\frac{2}{3 (2+\pi )} \Big(\frac{1}{(n_1-1){}^2}-6 (n_2-1) \sin \Big(\frac{\pi }{2(n_2-1)}\Big)+3 \pi \Big).\]
\end{prop} 
 
 \begin{proof}
By the hypothesis,  $\te{card}(\ga_n\ii L_1)=n_1$  and $\te{card}(\ga_n\ii L_2)=n_2$ for some $n_1, n_2\geq 2$. The proof of 
\[\ga_n\ii L_1=\set{(-1+\frac{2(j-1)}{n_1-1}, 0) : 1\leq j\leq n_1}\]
directly follows from Proposition~\ref{Me0}. Moreover, as shown in Proposition~\ref{Me0}, the distortion error for the $n_1$ elements is obtained as 
\[V(P; \ga_n\ii L_1)=\frac 2{3(2+\pi)}  \frac {1}{(n_1-1)^2}.\]
Let $\ga_n\ii L_2=\set{(\cos \gq_j, \sin \gq_j) : 1\leq j\leq n_2}$,
where $\gq_1=0$, and $\gq_{n_2}=\pi$. 
$L_2$ being a circular arc and $P$ is a uniform distribution, we have 
\[\gq_2-\gq_1=\gq_3-\gq_2=\cdots=\gq_{n_2}-\gq_{n_2-1}=\frac{\gq_{n_2}-\gq_1}{n_2-1}=\frac \pi {n_2-1}.\]
Thus, proceeding in a similar way as Proposition~\ref{Me0}, we have 
$\gq_j=\frac{(j-1)\pi}{n_2-1}$ for $1\leq j\leq n_2$ yielding 
\[\ga_n\ii L_2=\set{(\cos \frac{(j-1)\pi}{n_2-1}, \sin\frac {(j-1)\pi}{n_2-1}) : 1\leq j\leq n_2}.\]
The distortion error due to these $n_2$ elements is obtained as 
\begin{align*}
&V(P; \ga_n\ii L_2)\\
&=\frac 1{2+\pi} \Big(2\int_{0}^{\frac{\gq_1+\gq_2}{2}} \rho((\cos t, \sin t), (1, 0))\,dt+(n_2-2) \int_{\frac  {\gq_1+\gq_2}2}^{\frac  {\gq_2+\gq_3}2} \rho((\cos t, \sin t), (\cos \gq_2, \sin \gq_2))\,dt\Big)\\
&=\frac{2}{3 (2+\pi )} \Big(-6 (n_2-1) \sin \Big(\frac{\pi }{2(n_2-1)}\Big)+3 \pi \Big).
\end{align*} 
Thus, we have 
\begin{align*}
V_n&=V_{n_1, n_2}(P)=V(P; \ga_n\ii L_1)+V(P; \ga_n\ii L_2)\\
&=  \frac{2}{3 (2+\pi )} \Big(\frac{1}{(n_1-1){}^2}-6 (n_2-1) \sin \Big(\frac{\pi }{2(n_2-1)}\Big)+3 \pi \Big).
\end{align*} 
Hence, the proof of the proposition is complete. 
\end{proof}

\begin{remark}
Proposition~\ref{Me1} plays a significant role in calculating the optimal sets of $n$-points for the probability distribution $P$ on the boundary of the semicircular disc, as shown in the following two propositions. Recall that any element $x\in \D R$ can be identified as an element $(x, 0)\in \D R^2$.
\end{remark} 
\begin{prop}\label{Me200} 
 The conditional optimal set of three-points is given by $\set{(1, 0), (0, 1), (-1, 0)}$ with conditional constrained quantization error $V_3=\frac 2{2+\pi}(-2 \sqrt{2}+\frac{1}{3}+\pi)$. 
 \end{prop} 
\begin{proof}
Let $\ga$ be an optimal set of three-points. By the hypothesis $(1, 0), (-1 ,0)\in \ga$. Due to maximum symmetry, we can assume that either $(0, 0)$ or $(0, 1) \in \ga$. If $(0, 0) \in \ga$, then as $n_1=3$ and $n_2=2$, using Proposition~\ref{Me1}, we have the distortion error as 
$V_{3, 2}(P)=0.476477.$
If $(0, 1) \in\ga$, then as $n_1=2$ and $n_2=3$,  we have $V_{2, 3}(P)=0.251478$. Since $V_{2, 3}(P)<V_{3, 2}(P)$, the conditional optimal set of three-points is given by $\set{(1, 0), (0, 1), (-1, 0)}$ with conditional constrained quantization error $V_3=\frac 2{2+\pi}(-2 \sqrt{2}+\frac{1}{3}+\pi)$. 
Thus, the proof of the proposition is complete (see Figure~\ref{Fig1}). 
\end{proof}

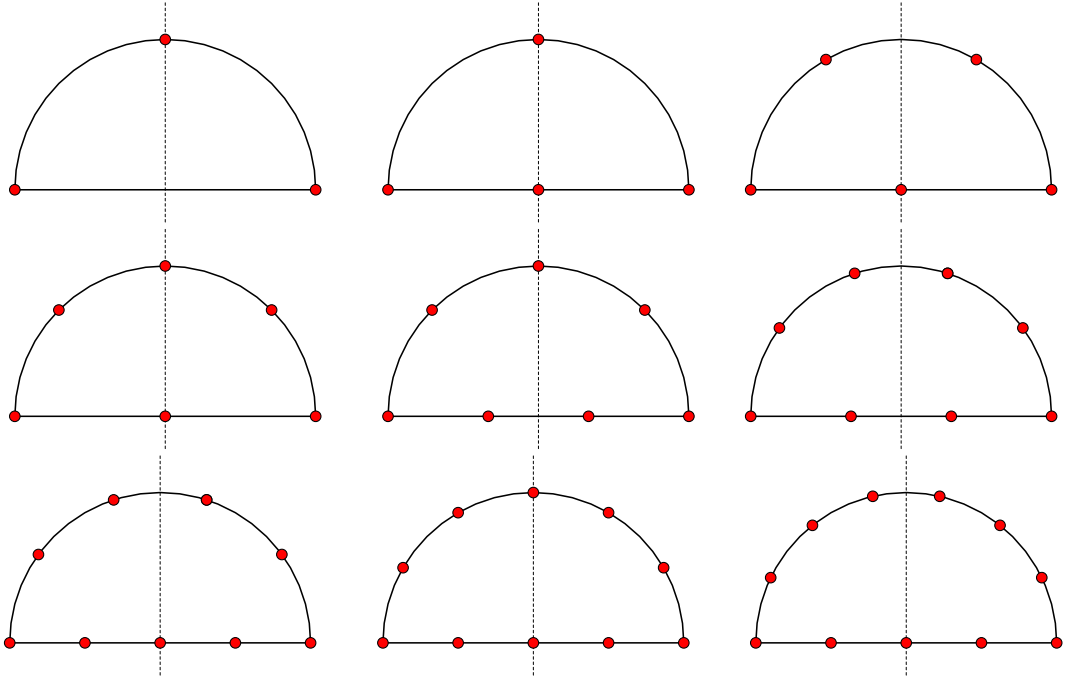
\begin{figure}
\begin{tikzpicture}[line cap=round,line join=round,>=triangle 45, x=1.0cm,y=1.0cm]
\clip(-2.1012214842676586,-0.438179149664134) rectangle (2.308087482904808,2.5425819497016167);
\draw [shift={(0.,0.)},line width=0.59 pt]  plot[domain=0.:3.141592653589793,variable=\t]({1.*2.*cos(\t r)+0.*2.*sin(\t r)},{0.*2.*cos(\t r)+1.*2.*sin(\t r)});
\draw [line width=0.59 pt ] (-2.,0.)-- (2.,0.);
\draw [line width=0.2 pt,dash pattern=on 1pt off 1pt] (0.,-0.5)-- (0.,2.5);
\begin{scriptsize}
\draw [fill=ffqqqq] (0, 2) circle (2.0pt);
\draw [fill=ffqqqq] (2, 0) circle (2.0pt);
\draw [fill=ffqqqq] (-2, 0) circle (2.0pt);
\end{scriptsize}
\end{tikzpicture}
\quad 
\begin{tikzpicture}[line cap=round,line join=round,>=triangle 45, x=1.0cm,y=1.0cm]
\clip(-2.1012214842676586,-0.438179149664134) rectangle (2.308087482904808,2.5425819497016167);
\draw [shift={(0.,0.)},line width=0.59 pt]  plot[domain=0.:3.141592653589793,variable=\t]({1.*2.*cos(\t r)+0.*2.*sin(\t r)},{0.*2.*cos(\t r)+1.*2.*sin(\t r)});
\draw [line width=0.59 pt ] (-2.,0.)-- (2.,0.);
\draw [line width=0.2 pt,dash pattern=on 1pt off 1pt] (0.,-0.5)-- (0.,2.5);
\begin{scriptsize}
\draw [fill=ffqqqq] (0, 2) circle (2.0pt);
\draw [fill=ffqqqq] (2, 0) circle (2.0pt);
\draw [fill=ffqqqq] (-2, 0) circle (2.0pt);
\draw [fill=ffqqqq] (0, 0) circle (2.0pt);
\end{scriptsize}
\end{tikzpicture}\quad 
\begin{tikzpicture}[line cap=round,line join=round,>=triangle 45, x=1.0cm,y=1.0cm]
\clip(-2.1012214842676586,-0.438179149664134) rectangle (2.308087482904808,2.5425819497016167);
\draw [shift={(0.,0.)},line width=0.59 pt]  plot[domain=0.:3.141592653589793,variable=\t]({1.*2.*cos(\t r)+0.*2.*sin(\t r)},{0.*2.*cos(\t r)+1.*2.*sin(\t r)});
\draw [line width=0.59 pt ] (-2.,0.)-- (2.,0.);
\draw [line width=0.2 pt,dash pattern=on 1pt off 1pt] (0.,-0.5)-- (0.,2.5);
\begin{scriptsize}
\draw [fill=ffqqqq] (-2,0) circle (2.0pt);
\draw [fill=ffqqqq] (0,0) circle (2.0pt);
\draw [fill=ffqqqq] (2,0) circle (2.0pt);
\draw [fill=ffqqqq] (1,1.73205) circle (2.0pt);
\draw [fill=ffqqqq] (-1,1.73205) circle (2.0pt);
\end{scriptsize}
\end{tikzpicture}

 \begin{tikzpicture}[line cap=round,line join=round,>=triangle 45, x=1.0cm,y=1.0cm]
\clip(-2.1012214842676586,-0.438179149664134) rectangle (2.308087482904808,2.5425819497016167);
\draw [shift={(0.,0.)},line width=0.59 pt]  plot[domain=0.:3.141592653589793,variable=\t]({1.*2.*cos(\t r)+0.*2.*sin(\t r)},{0.*2.*cos(\t r)+1.*2.*sin(\t r)});
\draw [line width=0.59 pt ] (-2.,0.)-- (2.,0.);
\draw [line width=0.2 pt,dash pattern=on 1pt off 1pt] (0.,-0.5)-- (0.,2.5);
\begin{scriptsize}
\draw [fill=ffqqqq] (-2,0) circle (2.0pt);
\draw [fill=ffqqqq] (0,0) circle (2.0pt);
\draw [fill=ffqqqq] (2,0) circle (2.0pt);
\draw [fill=ffqqqq] (0, 2) circle (2.0pt);
\draw [fill=ffqqqq] (1.41421, 1.41421) circle (2.0pt);
\draw [fill=ffqqqq] (-1.41421, 1.41421) circle (2.0pt);
\end{scriptsize}
\end{tikzpicture}
\quad 
\begin{tikzpicture}[line cap=round,line join=round,>=triangle 45, x=1.0cm,y=1.0cm]
\clip(-2.1012214842676586,-0.438179149664134) rectangle (2.308087482904808,2.5425819497016167);
\draw [shift={(0.,0.)},line width=0.59 pt]  plot[domain=0.:3.141592653589793,variable=\t]({1.*2.*cos(\t r)+0.*2.*sin(\t r)},{0.*2.*cos(\t r)+1.*2.*sin(\t r)});
\draw [line width=0.59 pt ] (-2.,0.)-- (2.,0.);
\draw [line width=0.2 pt,dash pattern=on 1pt off 1pt] (0.,-0.5)-- (0.,2.5);
\begin{scriptsize}
\draw [fill=ffqqqq] (-2,0) circle (2.0pt);
\draw [fill=ffqqqq] (-0.666667, 0.) circle (2.0pt);
\draw [fill=ffqqqq] (0.666667, 0.) circle (2.0pt);
\draw [fill=ffqqqq] (2,0) circle (2.0pt);
\draw [fill=ffqqqq] (0, 2) circle (2.0pt);
\draw [fill=ffqqqq] (1.41421, 1.41421) circle (2.0pt);
\draw [fill=ffqqqq] (-1.41421, 1.41421) circle (2.0pt);
\end{scriptsize}
\end{tikzpicture}\quad 
\begin{tikzpicture}[line cap=round,line join=round,>=triangle 45, x=1.0cm,y=1.0cm]
\clip(-2.1012214842676586,-0.438179149664134) rectangle (2.308087482904808,2.5425819497016167);
\draw [shift={(0.,0.)},line width=0.59 pt]  plot[domain=0.:3.141592653589793,variable=\t]({1.*2.*cos(\t r)+0.*2.*sin(\t r)},{0.*2.*cos(\t r)+1.*2.*sin(\t r)});
\draw [line width=0.59 pt ] (-2.,0.)-- (2.,0.);
\draw [line width=0.2 pt,dash pattern=on 1pt off 1pt] (0.,-0.5)-- (0.,2.5);
\begin{scriptsize}
\draw [fill=ffqqqq] (-2,0) circle (2.0pt);
\draw [fill=ffqqqq] (-0.666667, 0.) circle (2.0pt);
\draw [fill=ffqqqq] (0.666667, 0.) circle (2.0pt);
\draw [fill=ffqqqq] (2,0) circle (2.0pt);
\draw [fill=ffqqqq] (1.61803, 1.17557) circle (2.0pt);
\draw [fill=ffqqqq] (-1.61803, 1.17557) circle (2.0pt);
\draw [fill=ffqqqq] (0.618034, 1.90211) circle (2.0pt);
\draw [fill=ffqqqq] (-0.618034, 1.90211) circle (2.0pt);
\draw [fill=ffqqqq] (0.618034, 1.90211) circle (2.0pt);
\end{scriptsize}
\end{tikzpicture}

\begin{tikzpicture}[line cap=round,line join=round,>=triangle 45, x=1.0cm,y=1.0cm]
\clip(-2.1012214842676586,-0.438179149664134) rectangle (2.308087482904808,2.5425819497016167);
\draw [shift={(0.,0.)},line width=0.59 pt]  plot[domain=0.:3.141592653589793,variable=\t]({1.*2.*cos(\t r)+0.*2.*sin(\t r)},{0.*2.*cos(\t r)+1.*2.*sin(\t r)});
\draw [line width=0.59 pt ] (-2.,0.)-- (2.,0.);
\draw [line width=0.2 pt,dash pattern=on 1pt off 1pt] (0.,-0.5)-- (0.,2.5);
\begin{scriptsize}
\draw [fill=ffqqqq] (-2,0) circle (2.0pt);
\draw [fill=ffqqqq] (0, 0.) circle (2.0pt);
\draw [fill=ffqqqq] (-1, 0.) circle (2.0pt);
\draw [fill=ffqqqq] (1, 0.) circle (2.0pt);
\draw [fill=ffqqqq] (2,0) circle (2.0pt);
\draw [fill=ffqqqq] (1.61803, 1.17557) circle (2.0pt);
\draw [fill=ffqqqq] (-1.61803, 1.17557) circle (2.0pt);
\draw [fill=ffqqqq] (0.618034, 1.90211) circle (2.0pt);
\draw [fill=ffqqqq] (-0.618034, 1.90211) circle (2.0pt);
\draw [fill=ffqqqq] (0.618034, 1.90211) circle (2.0pt);
\end{scriptsize}
\end{tikzpicture}
\quad 
\begin{tikzpicture}[line cap=round,line join=round,>=triangle 45, x=1.0cm,y=1.0cm]
\clip(-2.1012214842676586,-0.438179149664134) rectangle (2.308087482904808,2.5425819497016167);
\draw [shift={(0.,0.)},line width=0.59 pt]  plot[domain=0.:3.141592653589793,variable=\t]({1.*2.*cos(\t r)+0.*2.*sin(\t r)},{0.*2.*cos(\t r)+1.*2.*sin(\t r)});
\draw [line width=0.59 pt ] (-2.,0.)-- (2.,0.);
\draw [line width=0.2 pt,dash pattern=on 1pt off 1pt] (0.,-0.5)-- (0.,2.5);
\begin{scriptsize}
\draw [fill=ffqqqq] (-2,0) circle (2.0pt);
\draw [fill=ffqqqq] (0, 0.) circle (2.0pt);
\draw [fill=ffqqqq] (-1, 0.) circle (2.0pt);
\draw [fill=ffqqqq] (1, 0.) circle (2.0pt);
\draw [fill=ffqqqq] (2,0) circle (2.0pt);
\draw [fill=ffqqqq] (1.73205, 1) circle (2.0pt);
\draw [fill=ffqqqq] (1., 1.73205) circle (2.0pt);
\draw [fill=ffqqqq] (0., 2) circle (2.0pt);
\draw [fill=ffqqqq] (-1., 1.73205) circle (2.0pt);
\draw [fill=ffqqqq] (-1.73205, 1.) circle (2.0pt);
\end{scriptsize}
\end{tikzpicture}
\quad 
\begin{tikzpicture}[line cap=round,line join=round,>=triangle 45, x=1.0cm,y=1.0cm]
\clip(-2.1012214842676586,-0.438179149664134) rectangle (2.308087482904808,2.5425819497016167);
\draw [shift={(0.,0.)},line width=0.59 pt]  plot[domain=0.:3.141592653589793,variable=\t]({1.*2.*cos(\t r)+0.*2.*sin(\t r)},{0.*2.*cos(\t r)+1.*2.*sin(\t r)});
\draw [line width=0.59 pt ] (-2.,0.)-- (2.,0.);
\draw [line width=0.2 pt,dash pattern=on 1pt off 1pt] (0.,-0.5)-- (0.,2.5);
\begin{scriptsize}
\draw [fill=ffqqqq] (-2,0) circle (2.0pt);
\draw [fill=ffqqqq] (0, 0.) circle (2.0pt);
\draw [fill=ffqqqq] (-1, 0.) circle (2.0pt);
\draw [fill=ffqqqq] (1, 0.) circle (2.0pt);
\draw [fill=ffqqqq] (2,0) circle (2.0pt);
\draw [fill=ffqqqq] (1.80194, 0.867767) circle (2.0pt);
\draw [fill=ffqqqq] (1.24698, 1.56366) circle (2.0pt);
\draw [fill=ffqqqq] (0.445042, 1.94986) circle (2.0pt);
\draw [fill=ffqqqq] (-0.445042, 1.94986) circle (2.0pt);
\draw [fill=ffqqqq] (-1.24698, 1.56366) circle (2.0pt);
\draw [fill=ffqqqq] (-1.80194, 0.867767) circle (2.0pt);
\end{scriptsize}
\end{tikzpicture}
 \caption{Optimal configuration of $n$-points for $3\leq n\leq 11$.} \label{Fig1}
\end{figure}

\begin{prop}\label{Me300} 
 The conditional optimal set of four-points is given by $\set{(0, 0), (1, 0), (0, 1), (-1, 0)}$ with conditional constrained quantization error $V_4=\frac{-24 \sqrt{2}+12 \pi +1}{12+6 \pi }$. 
 \end{prop} 
\begin{proof}
Let $\ga$ be an optimal set of four-points. By the hypothesis $(1, 0), (-1 ,0)\in \ga$. Due to maximum symmetry, we can assume that the other two elements in $\ga$ are on the axis of symmetry, or they are symmetrically located on $L$. Thus, the following cases can occur: 

\tit{Case~1. The two elements in $\ga\setminus \set {(1, 0), (-1, 0)}$ are on the axis of symmetry.}

In this case, we can assume that $(0, 0) , (0, 1)\in \ga$. Then, as $n_1=3$ and $n_2=3$, the distortion error is $V_{3,3}(P)=0.154232.$

\tit{Case~2. The two elements in $\ga\setminus \set {(1, 0), (-1, 0)}$ are symmetrically located on $L$.} 

In this case, the two elements are either symmetrically located on $L_1$ or on $L_2$. If they are symmetrically located on $L_1$, then 
the distortion error is $V_{4, 2}(P)=0.458469$.  If they are symmetrically located on $L_2$, then the distortion error is $V_{2, 4}(P)=0.184739.$

Considering all the above possible distortion errors, we see that the distortion error is minimum when $n_1=3$ and $n_2=3$. Thus, $\set{(0, 0), (1, 0), (0, 1), (-1, 0)}$ forms the conditional optimal set of four-points with conditional constrained quantization error $V_4=\frac{-24 \sqrt{2}+12 \pi +1}{12+6 \pi }$, which is the proposition (see Figure~\ref{Fig1}). 
\end{proof}

Let us now give the following theorem, which gives the main result in this section. 
\begin{theo}\label{theo1}
Let $\ga_n$ be a conditional optimal set of $n$-points for $P$ for some $n\geq 3$. Let $\te{card}(\ga_n\ii L_1)=n_1$ be known. Then, 
\begin{align*}
\ga_n=\set{(-1+\frac{2(j-1)}{n_1-1}, 0) : 1\leq j\leq n_1}\UU \set{(\cos \frac{(j-1)\pi}{n-n_1+1}, \sin\frac {(j-1)\pi}{n-n_1+1}) : 2\leq j\leq n-n_1+1} 
\end{align*}
with the conditional constrained quantization error 
\[V_n=\frac{2}{3 (2+\pi )} \Big(\frac{1}{(n_1-1){}^2}-6 (n-n_1+1) \sin (\frac{\pi }{2 n-2 n_1+2})+3 \pi \Big).\]
 \end{theo}
\begin{proof}
Let $\ga_n$ be a conditional optimal set of $n$-points for $P$. Let $\te{card}(\ga_n\ii L_1)=n_1$  and $\te{card}(\ga_n\ii L_2)=n_2$. Notice that $\ga_n\ii L_1\ii L_2=\set{(-1, 0), (0, 1)}$, and hence $n_1+n_2=n+2$ yielding $n_2=n-n_1+2$. Thus, if $n_1$ is known, one can easily calculate $n_2$, and then the conditional optimal set $\ga_n$ of $n$-points and the corresponding conditional constrained quantization error can be deduced by Proposition~\ref{Me1}, and they are given by 
\begin{align*}
\ga_n=\set{(-1+\frac{2(j-1)}{n_1-1}, 0) : 1\leq j\leq n_1}\UU \set{(\cos \frac{(j-1)\pi}{n_2-1}, \sin\frac {(j-1)\pi}{n_2-1}) : 2\leq j\leq n-n_1+1} 
\end{align*} 
with $V_n=V_{n_1, n-n_1+2}=\frac{2}{3 (2+\pi )}\Big(\frac{1}{(n_1-1){}^2}-6 (n-n_1+1) \sin (\frac{\pi }{2 n-2 n_1+2})+3 \pi \Big)$.
Thus, the proof of the theorem is complete. 
\end{proof}

For a given positive integer $n\geq 3$, to determine the positive integer $n_1$ as mentioned in Theorem~\ref{theo1}, we proceed as follows: 
 
\begin{defi} \label{difi21}
Define the sequence $\set{a(n)}$ such that $a(n)=\lfloor\frac{5 (n+4)}{13}\rfloor$ for $n\geq 1$, i.e.,
\begin{align*}
 \set{a(n)}_{n=1}^\infty=&\set{1, 2, 2, 3, 3, 3, 4, 4, 5, 5, 5, 6, 6, 6, 7, 7, 8, 8, 8, 9, 9, 10, 10, 10, 11, 11, 11, 12, 12, 13, 13, 13, \cdots},
\end{align*}
where $\lfloor x\rfloor$ represents the greatest integer not exceeding $x$.
\end{defi}

The following algorithm helps us to determine the exact value of $n_1$ mentioned in Theorem~\ref{theo1}.
\subsection{Algorithm}  Let $n\geq 3$, and let $V_n:=V_{n_1, n-n_1+2}$, as given by Theorem~\ref{theo1}, denote the distortion error if an optimal set $\ga_n$ contains $n_1$ elements from the base of the semicircular disc. Write $F(n, n_1):=V_{n_1, n-n_1+2}$. Let $\set{a(n)}$ be the sequence  defined by Definition~\ref{difi21}. Then, the algorithm runs as follows:

$(i)$ Write $n_1:=a(n)$.

$(ii)$ If $n_1=2$ go to step $(v)$, else step $(iii)$.

$(iii)$ If $F(n, n_1-1)<F(n, n_1)$ replace $n_1$ by $n_1-1$ and go to step $(ii)$, else step $(iv)$.

$(iv)$ If $F(n, n_1+1)<F(n, n_1)$ replace $n_1$ by $n_1+1$ and return, else step $(v)$.

$(v)$ End.

When the algorithm ends, then the value of $n_1$, obtained, is the exact value of $n_1$ that an optimal set $\ga_n$ contains from the base of the semicircular disc.
\begin{remark}
If $n=50$, then $a(n)=20$, and by the algorithm we also obtain $n_1=20$; if $n=80$, then $a(n)=32$, and by the algorithm we also obtain $n_1=32$. If $n=1200$, then  $a(n)=463$, and by the algorithm, we obtain $n_1=468$; if $n=2000$, then $a(n)=770$, and by the algorithm, we obtain $n_1=779$; and if $n=3000$, then $a(n)=1155$, and by the algorithm, we obtain $n_1=1168$. Thus, we see that with the help of the sequence and the algorithm, one can easily determine the exact value of $n_1$ for any positive integer $n\geq 3$.
\end{remark}

 \section{Conditional unconstrained quantization for a uniform distribution on an equilateral triangle} \label{secMe4}

Let $\tri OAB$ be an equilateral triangle with vertices $O(0,0), A(1, 0), B(\frac 12,\frac{\sqrt 3}2)$. Let $L_1, L_2, L_3$ be the sides $OA, AB$ and $BO$, respectively. Let $P$ be the uniform distribution defined on the equilateral triangle $\tri$ formed by the sides $L_1, L_2, L_3$. Let $s$ represent the distance of any point on $\tri$ from the origin tracing along the boundary of the triangle in the counterclockwise direction. Then, the points $O, A, B$ are, respectively, represented by $s=0, s=1, s=2$. 
 The probability density function (pdf) $f$ of the uniform distribution $P$ is given by $f(s):=f(x_1, x_2)=\frac 1{3}$ for all $(x_1, x_2)\in L_1\uu L_2\uu L_3$, and zero otherwise. The sides $L_1, L_2, L_3$ are represented by the parametric equations as follows:
\begin{align*}
L_1&=\set {(t, 0)},  \quad L_2=\Big\{\Big (1-\frac t 2,\frac{\sqrt{3} t}{2}\Big)\Big\}, \quad L_3=\Big\{\Big(\frac{1-t}{2}, \frac{\sqrt{3}(1-t)}{2}\Big)\Big\},
\end{align*}
where $0\leq t\leq 1$. 
Again, $dP(s)=P(ds)=f(x_1, x_2) ds$. On each $L_j$ for $1\leq j\leq 3$, we have $(ds)^2=(dx_1)^2+(dx_2)^2=(dt)^2$ yielding $ds=dt$. In the definition of conditional unconstrained quantization error, let $\gb:=\set{(0, 0), (1, 0), (\frac 12, \frac {\sqrt {3}} 2)}$. Let $L:=L_1\uu L_2\uu L_3$ which is the support of $P$. Upon the given set $\gb$, the $n$th conditional unconstrained quantization errors are defined for all $n\geq 3$.
Notice that the conditional optimal set of three-points is the set $\gb$ itself. In the sequel of this section, we investigate the conditional optimal sets of $n$-points for $n\geq 4$. 

 Consider the following two affine transformations: 
\begin{equation} \label{eqM1} 
T_1(x, y)=(-\frac 12 x+1, \frac{\sqrt 3}{2} x) \te{ and } T_2(x, y)=(-\frac 12 x+\frac 12, -\frac{\sqrt 3}2 x+\frac {\sqrt 3} 2).
\end{equation} 

Let us now give the following proposition, which plays a vital role in this section. In this regard, one can also see \cite{RR}.

\begin{prop} \label{Me2} 
Let $\ga_n$ be a conditional optimal set of $n$-points for $P$ for some $n\geq 4$. Let $\te{card}(\ga_n\ii L_1)=n_1$, $\te{card}(\ga_n\ii L_2)=n_2$ and $\te{card}(\ga_n\ii L_3)=n_3$  with corresponding conditional unconstrained quantization error $V_n:=V_{n_1, n_2, n_3}(P)$ for some $n_1, n_2, n_3\geq 2$. Then, 
\begin{align*}
\ga_n\ii L_1& =\Big\{(\frac{j-1}{n_1-1}, 0) : 1\leq j\leq n_1\Big\},\\
 \ga_n\ii L_2&=\Big\{T_1(\frac{j-1}{n_2-1}, 0) : 1\leq j\leq n_2\Big\}, \te{ and } \\
 \ga_n\ii L_3&=\Big\{T_2(\frac{j-1}{n_3-1}, 0) : 1\leq j\leq n_3\Big\},
\end{align*}
with  
\[V_{n_1, n_2, n_3}(P)=\frac{1}{36} \Big(\frac 1{(n_1-1)^2}+\frac 1{(n_2-1)^2}+\frac 1{(n_3-1)^2}\Big).\]
\end{prop} 
 
 \begin{proof}
By the hypothesis,  $\te{card}(\ga_n\ii L_1)=n_1$ for some $n_1\geq 2$.
Let \begin{equation}
\ga_n\ii L_1=\set{(a_j, 0) : 1\leq j\leq n_1},
\end{equation} 
By Proposition~\ref{Me0}, we have $a_j=\frac{j-1}{n_1-1}$ for $1\leq j\leq n_1$ implying 
\[\ga_n\ii L_1=\set{(\frac{j-1}{n_1-1}, 0) : 1\leq j\leq n_1}.\] 
Given $\te{card}(\ga_n\ii L_2)=n_2$. If $\te{card}(\ga_n\ii L_1)=n_2$, then as before we have 
\[\ga_n\ii L_1=\set{(\frac{j-1}{n_2-1}, 0) : 1\leq j\leq n_2}.\] 
Hence, $T_2$ being an affine transformation such that $T_2(L_1)=L_2$, we have 
\[\ga_n\ii L_2=\set{T_1(\frac{j-1}{n_2-1}, 0) : 1\leq j\leq n_2}.\] 
Similarly, we have 
\[\ga_n\ii L_3=\set{T_2(\frac{j-1}{n_3-1}, 0) : 1\leq j\leq n_3}.\] 
To find the quantization error, we proceed as follows.  Let $V(P; \ga_n\ii L_j)$ denote the distortion errors contributed by the elements in $\ga_n\ii L_j$ for $j=1,2,3$.
Then,  \[V_{n_1, n_2, n_3}(P)=V(P; \ga_n\ii L_1)+V(P; \ga_n\ii L_2)+V(P; \ga_n\ii L_3).\]
By Proposition~\ref{Me0}, we have
\begin{align*}
V(P; \ga_n\ii L_1)=\frac{1}{36 \left(n_1-1\right){}^2}.
\end{align*} 
Due to rotational symmetry, we have 
\begin{align*}
V(P; \ga_n\ii L_2)=\frac{1}{36 \left(n_2-1\right){}^2} \te{ and } V(P; \ga_n\ii L_3)=\frac{1}{36 \left(n_3-1\right){}^2}.
\end{align*} 
Hence, 
\[V_{n_1, n_2, n_3}(P)=\frac{1}{36} \Big(\frac 1{(n_1-1)^2}+\frac 1{(n_2-1)^2}+\frac 1{(n_3-1)^2}\Big).\]
Thus, the proof of the proposition is complete. 
\end{proof} 

\begin{note} \label{note1}
Let $\ga_n(L_j)$ be the set consisting of all the elements in $\ga_n\ii L_j$ except the right endpoint for each $j=1, 2, 3$. Then, 
\begin{align*} 
\ga_n(L_1)& =\Big\{(\frac{j-1}{n_1-1}, 0) : 1\leq j\leq n_1-1\Big\},\\
\ga_n(L_2)&=\Big\{T_1(\frac{j-1}{n_2-1}, 0) : 1\leq j\leq n_2-1\Big\}, \te{ and } \\
\ga_n(L_3)&=\Big\{T_2(\frac{j-1}{n_3-1}, 0) : 1\leq j\leq n_3-1\Big\}
\end{align*}
with $\te{card}(\ga_n(L_1))=n_1-1$, $\te{card}(\ga_n(L_2))=n_2-1$, and $\te{card}(\ga_n(L_3))=n_3-1$. Notice that the sets $\ga_n(L_j)$ are disjoints and $(n_1-1)+(n_2-1)+(n_3-1)=n$. 
\end{note}

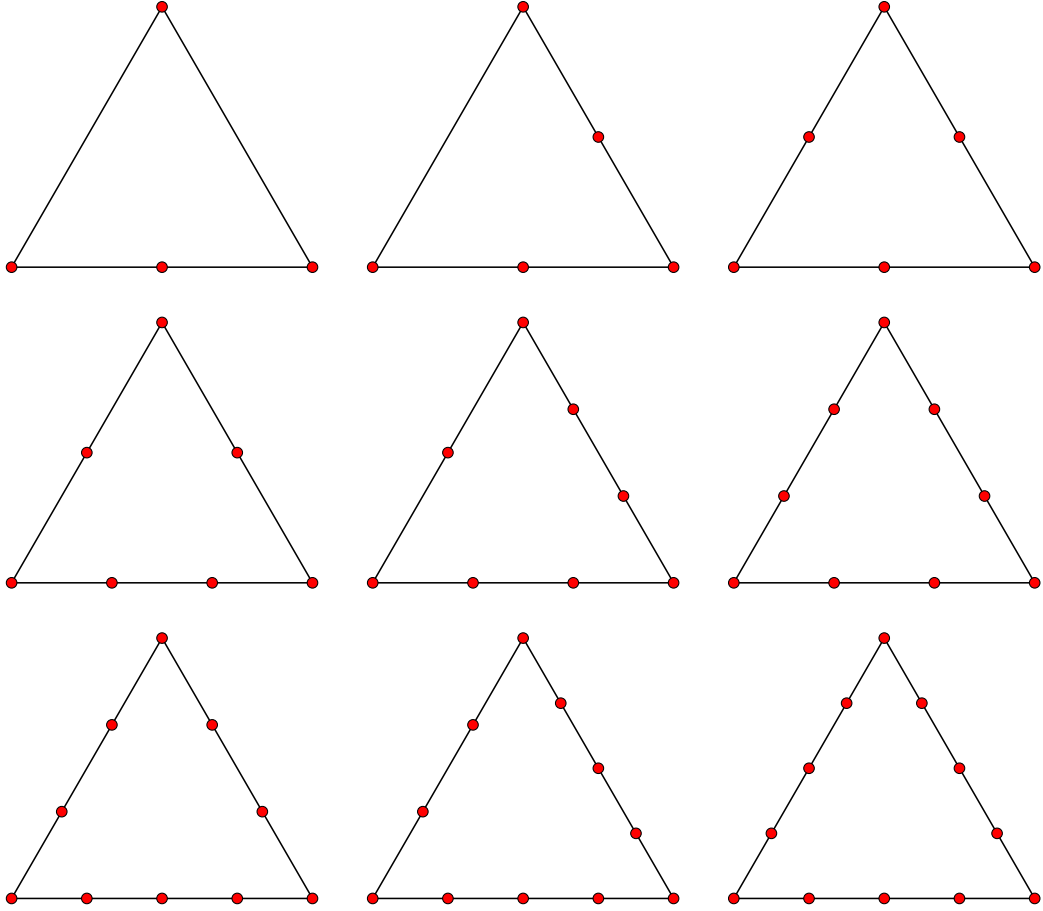
\begin{figure}
\begin{tikzpicture}[line cap=round,line join=round,>=triangle 45,x=1.0cm,y=1.0cm]
\clip(-0.1275878974887044,-0.153549597395171) rectangle (4.123200541670084,4.0114443960552855);
\draw [line width=0.59 pt] (0.,0.)-- (4.,0.);
\draw [line width=0.59 pt] (0.,0.)-- (2., 3.4641);
\draw [line width=0.59 pt] (4.,0.)-- (2., 3.4641);
\begin{scriptsize}
\draw [fill=ffqqqq] (0, 0) circle (2.0 pt);
\draw [fill=ffqqqq] (4.,0.) circle (2.0 pt);
\draw [fill=ffqqqq] (2., 3.4641) circle (2.0 pt);
\draw [fill=ffqqqq] (2.,0.) circle (2.0 pt);
\end{scriptsize}
\end{tikzpicture} 
\quad 
\begin{tikzpicture}[line cap=round,line join=round,>=triangle 45,x=1.0cm,y=1.0cm]
\clip(-0.1275878974887044,-0.153549597395171) rectangle (4.123200541670084,4.0114443960552855);
\draw [line width=0.59 pt] (0.,0.)-- (4.,0.);
\draw [line width=0.59 pt] (0.,0.)-- (2., 3.4641);
\draw [line width=0.59 pt] (4.,0.)-- (2., 3.4641);
\begin{scriptsize}
\draw [fill=ffqqqq] (0, 0) circle (2.0 pt);
\draw [fill=ffqqqq] (4.,0.) circle (2.0 pt);
\draw [fill=ffqqqq] (2., 3.4641) circle (2.0 pt);
\draw [fill=ffqqqq] (2.,0.) circle (2.0 pt);
\draw [fill=ffqqqq] (3., 1.73205) circle (2.0 pt);
\end{scriptsize}
\end{tikzpicture}  \quad 
\begin{tikzpicture}[line cap=round,line join=round,>=triangle 45,x=1.0cm,y=1.0cm]
\clip(-0.1275878974887044,-0.153549597395171) rectangle (4.123200541670084,4.0114443960552855);
\draw [line width=0.59 pt] (0.,0.)-- (4.,0.);
\draw [line width=0.59 pt] (0.,0.)-- (2., 3.4641);
\draw [line width=0.59 pt] (4.,0.)-- (2., 3.4641);
\begin{scriptsize}
\draw [fill=ffqqqq] (0, 0) circle (2.0 pt);
\draw [fill=ffqqqq] (4.,0.) circle (2.0 pt);
\draw [fill=ffqqqq] (2., 3.4641) circle (2.0 pt);
\draw [fill=ffqqqq] (2.,0.) circle (2.0 pt);
\draw [fill=ffqqqq] (3., 1.73205) circle (2.0 pt);
\draw [fill=ffqqqq] (1., 1.73205) circle (2.0 pt);
\end{scriptsize}
\end{tikzpicture} 

\begin{tikzpicture}[line cap=round,line join=round,>=triangle 45,x=1.0cm,y=1.0cm]
\clip(-0.1275878974887044,-0.153549597395171) rectangle (4.123200541670084,4.0114443960552855);
\draw [line width=0.59 pt] (0.,0.)-- (4.,0.);
\draw [line width=0.59 pt] (0.,0.)-- (2., 3.4641);
\draw [line width=0.59 pt] (4.,0.)-- (2., 3.4641);
\begin{scriptsize}
\draw [fill=ffqqqq] (0, 0) circle (2.0 pt);
\draw [fill=ffqqqq] (4.,0.) circle (2.0 pt);
\draw [fill=ffqqqq] (2., 3.4641) circle (2.0 pt);
\draw [fill=ffqqqq] (1.33333, 0) circle (2.0 pt);
\draw [fill=ffqqqq] (2.66667, 0) circle (2.0 pt);
\draw [fill=ffqqqq] (3., 1.73205) circle (2.0 pt);
\draw [fill=ffqqqq] (1., 1.73205) circle (2.0 pt);
\end{scriptsize}
\end{tikzpicture} 
\quad 
\begin{tikzpicture}[line cap=round,line join=round,>=triangle 45,x=1.0cm,y=1.0cm]
\clip(-0.1275878974887044,-0.153549597395171) rectangle (4.123200541670084,4.0114443960552855);
\draw [line width=0.59 pt] (0.,0.)-- (4.,0.);
\draw [line width=0.59 pt] (0.,0.)-- (2., 3.4641);
\draw [line width=0.59 pt] (4.,0.)-- (2., 3.4641);
\begin{scriptsize}
\draw [fill=ffqqqq] (0, 0) circle (2.0 pt);
\draw [fill=ffqqqq] (4.,0.) circle (2.0 pt);
\draw [fill=ffqqqq] (2., 3.4641) circle (2.0 pt);
\draw [fill=ffqqqq] (1.33333, 0) circle (2.0 pt);
\draw [fill=ffqqqq] (2.66667, 0) circle (2.0 pt);
\draw [fill=ffqqqq] (3.33333, 1.1547) circle (2.0 pt);
\draw [fill=ffqqqq] (2.66667, 2.3094) circle (2.0 pt);
\draw [fill=ffqqqq] (1., 1.73205) circle (2.0 pt);
\end{scriptsize}
\end{tikzpicture}  \quad 
\begin{tikzpicture}[line cap=round,line join=round,>=triangle 45,x=1.0cm,y=1.0cm]
\clip(-0.1275878974887044,-0.153549597395171) rectangle (4.123200541670084,4.0114443960552855);
\draw [line width=0.59 pt] (0.,0.)-- (4.,0.);
\draw [line width=0.59 pt] (0.,0.)-- (2., 3.4641);
\draw [line width=0.59 pt] (4.,0.)-- (2., 3.4641);
\begin{scriptsize}
\draw [fill=ffqqqq] (0, 0) circle (2.0 pt);
\draw [fill=ffqqqq] (4.,0.) circle (2.0 pt);
\draw [fill=ffqqqq] (2., 3.4641) circle (2.0 pt);
\draw [fill=ffqqqq] (1.33333, 0) circle (2.0 pt);
\draw [fill=ffqqqq] (2.66667, 0) circle (2.0 pt);
\draw [fill=ffqqqq] (3.33333, 1.1547) circle (2.0 pt);
\draw [fill=ffqqqq] (2.66667, 2.3094) circle (2.0 pt);
\draw [fill=ffqqqq] (1.33333, 2.3094) circle (2.0 pt);
\draw [fill=ffqqqq] (0.666667, 1.1547) circle (2.0 pt);
\end{scriptsize}
\end{tikzpicture} 

\begin{tikzpicture}[line cap=round,line join=round,>=triangle 45,x=1.0cm,y=1.0cm]
\clip(-0.1275878974887044,-0.153549597395171) rectangle (4.123200541670084,4.0114443960552855);
\draw [line width=0.59 pt] (0.,0.)-- (4.,0.);
\draw [line width=0.59 pt] (0.,0.)-- (2., 3.4641);
\draw [line width=0.59 pt] (4.,0.)-- (2., 3.4641);
\begin{scriptsize}
\draw [fill=ffqqqq] (0, 0) circle (2.0 pt);
\draw [fill=ffqqqq] (4.,0.) circle (2.0 pt);
\draw [fill=ffqqqq] (2., 3.4641) circle (2.0 pt);
\draw [fill=ffqqqq] (1.0, 0) circle (2.0 pt);
\draw [fill=ffqqqq] (2.0, 0) circle (2.0 pt);
\draw [fill=ffqqqq] (3.0, 0) circle (2.0 pt);
\draw [fill=ffqqqq] (3.33333, 1.1547) circle (2.0 pt);
\draw [fill=ffqqqq] (2.66667, 2.3094) circle (2.0 pt);
\draw [fill=ffqqqq] (1.33333, 2.3094) circle (2.0 pt);
\draw [fill=ffqqqq] (0.666667, 1.1547) circle (2.0 pt);
\end{scriptsize}
\end{tikzpicture} 
\quad 
\begin{tikzpicture}[line cap=round,line join=round,>=triangle 45,x=1.0cm,y=1.0cm]
\clip(-0.1275878974887044,-0.153549597395171) rectangle (4.123200541670084,4.0114443960552855);
\draw [line width=0.59 pt] (0.,0.)-- (4.,0.);
\draw [line width=0.59 pt] (0.,0.)-- (2., 3.4641);
\draw [line width=0.59 pt] (4.,0.)-- (2., 3.4641);
\begin{scriptsize}
\draw [fill=ffqqqq] (0, 0) circle (2.0 pt);
\draw [fill=ffqqqq] (4.,0.) circle (2.0 pt);
\draw [fill=ffqqqq] (2., 3.4641) circle (2.0 pt);
\draw [fill=ffqqqq] (1.0, 0) circle (2.0 pt);
\draw [fill=ffqqqq] (2.0, 0) circle (2.0 pt);
\draw [fill=ffqqqq] (3.0, 0) circle (2.0 pt);
\draw [fill=ffqqqq] (3.5, 0.866025) circle (2.0 pt);
\draw [fill=ffqqqq] (3., 1.73205) circle (2.0 pt);
\draw [fill=ffqqqq] (2.5, 2.59808) circle (2.0 pt);
\draw [fill=ffqqqq] (1.33333, 2.3094) circle (2.0 pt);
\draw [fill=ffqqqq] (0.666667, 1.1547) circle (2.0 pt);
\end{scriptsize}
\end{tikzpicture}  \quad 
\begin{tikzpicture}[line cap=round,line join=round,>=triangle 45,x=1.0cm,y=1.0cm]
\clip(-0.1275878974887044,-0.153549597395171) rectangle (4.123200541670084,4.0114443960552855);
\draw [line width=0.59 pt] (0.,0.)-- (4.,0.);
\draw [line width=0.59 pt] (0.,0.)-- (2., 3.4641);
\draw [line width=0.59 pt] (4.,0.)-- (2., 3.4641);
\begin{scriptsize}
\draw [fill=ffqqqq] (0, 0) circle (2.0 pt);
\draw [fill=ffqqqq] (4.,0.) circle (2.0 pt);
\draw [fill=ffqqqq] (2., 3.4641) circle (2.0 pt);
\draw [fill=ffqqqq] (1.0, 0) circle (2.0 pt);
\draw [fill=ffqqqq] (2.0, 0) circle (2.0 pt);
\draw [fill=ffqqqq] (3.0, 0) circle (2.0 pt);
\draw [fill=ffqqqq] (3.5, 0.866025) circle (2.0 pt);
\draw [fill=ffqqqq] (3., 1.73205) circle (2.0 pt);
\draw [fill=ffqqqq] (2.5, 2.59808) circle (2.0 pt);
\draw [fill=ffqqqq] (1.5, 2.59808) circle (2.0 pt);
\draw [fill=ffqqqq] (1., 1.73205) circle (2.0 pt);
\draw [fill=ffqqqq] (0.5, 0.866025) circle (2.0 pt);
\end{scriptsize}
\end{tikzpicture} 
 \caption{Optimal configuration of $n$-points for $4\leq n\leq 12$.} \label{Fig2}
\end{figure}

\begin{lemma} \label{lemmaMe1} 
Let $n_j\in \D N$ for $j=1, 2, 3$ be the numbers as defined in Proposition~\ref{Me2}. Then, $|n_i-n_j|=0$, or $1$ for $1\leq i\neq j\leq 3$. 
\end{lemma}

\begin{proof} We first show that $|n_1-n_2|=0, \te{ or } 1$. 
Write $m:=(n_1-1)+(n_2-1)$. Now, the distortion error contributed by the $m$ elements in $\ga_n(L_1)\uu \ga_n(L_2)$ is given by 
\[\frac{1}{36} \Big(\frac 1{(n_1-1)^2}+\frac 1{(n_2-1)^2}\Big).\]
The above expression is minimum if $n_1-1\approx \frac m 2$ and $n_2-1\approx \frac m 2$. Thus, we see that if $m=2k$ for some positive integer $k$, then $n_1-1=n_2-1=k$, and if $m=2k+1$ for some positive integer $k$, then either $(n_1-1=k+1$ and $n_2-1=k)$ or $(n_1-1=k$ and $n_2-1=k+1)$ which yield the fact that $|n_1-n_2|=0, \te{ or } 1$. Similarly, we can show that  $|n_i-n_j|=0$, or $1$ for any $1\leq i\neq j\leq 3$. Thus, the proof of the lemma is complete. 
\end{proof}  

Let us now give the following theorem, which gives the main result in this section. 
\begin{theo}\label{theo2}
Let $\ga_n$ be a conditional optimal set of $n$-points for $P$ for some $n\geq 4$. Then, 
\begin{align*}
\ga_n&=\Big\{(\frac{j-1}{n_1-1}, 0) : 1\leq j\leq n_1-1\Big\}\UU \Big\{T_1(\frac{j-1}{n_2-1}, 0) : 1\leq j\leq n_2-1\Big\}\\
&\qquad \qquad \UU\Big\{T_2(\frac{j-1}{n_3-1}, 0) : 1\leq j\leq n_3-1\Big\}
\end{align*} 
with the conditional unconstrained quantization error 
\[V_n=\frac{1}{36} \Big(\frac 1{(n_1-1)^2}+\frac 1{(n_2-1)^2}+\frac 1{(n_3-1)^2}\Big),\]
where $n_1, n_2, n_3$ are given as follows: For some $k\in \D N$, if $n=3k$, then $n_1-1=n_2-1=n_3-1=k$; if $n=3k+1$, then $n_1-1=k+1$ and $n_2-1=n_3-1=k$; and if $n=3k+2$, then $n_1-1=n_2-1=k+1$ and $n_3-1=k$. 
 \end{theo}
\begin{proof}
Let $\ga_n(L_j)$ be the sets defined by Note~\ref{note1} for $j=1, 2, 3$. Notice that for $j=1, 2, 3,$ the sets $\ga_n(L_j)$ are nonempty, and so we can find three positive integers $n_j\geq 2$ as defined in Proposition~\ref{Me2} such that $\te{card}(\ga_n(L_j))=n_j-1$.  Since 
\[\ga_n=\UU_{j=1}^3\ga_n(L_j),\]
the expression for $\ga_n$ follows by Note~\ref{note1}. The expression for the conditional unconstrained quantization error $V_n$ follows from Proposition~\ref{Me2}. By Lemma~\ref{lemmaMe1}, it follows that for some $k\in \D N$, if $n=3k$, then $n_1-1=n_2-1=n_3-1=k$; if $n=3k+1$, then $n_1-1=k+1$ and $n_2-1=n_3-1=k$; and if $n=3k+2$, then $n_1-1=n_2-1=k+1$ and $n_3-1=k$.  Thus, the proof of the theorem is complete. 
\end{proof}

Using Theorem~\ref{theo2}, the following example can be obtained. 
 \begin{exam} 
A conditional unconstrained optimal set of four-points is $\set{(0, 0), (\frac 12, 1), (1, 0), (\frac 12, \frac{\sqrt 3}2)}$ with conditional unconstrained quantization error $V_4=\frac 1 {16}$; a conditional unconstrained optimal set of five-points is $\set{(0, 0), (\frac 12, 1), (1, 0), (\frac 1 4, \frac{\sqrt 3}4), (\frac 12, \frac{\sqrt 3}2)}$ with conditional unconstrained quantization error $V_5=\frac 1 {24}$; and so on (see Figure~\ref{Fig2}).
 \end{exam} 
 
\begin{theorem}\label{theo3} 
The conditional unconstrained quantization dimension $D(P)$ of the probability measure $P$ exists, and $D(P)=1$, and the $D(P)$-dimensional conditional unconstrained quantization coefficient for $P$ exists as a finite positive number and equals $\frac 34$. 
\end{theorem}

\begin{proof}
For $n\in \D N$ with $n\geq 4$, let $\ell(n)$ be the unique natural number such that $3{\ell(n)}\leq n<3 ({\ell(n)+1})$. Then,
$V_{3(\ell(n)+1)}\leq V_n\leq V_{3{\ell(n)}}$, i.e., $\frac 1 {12(\ell(n)+1)^2}\leq V_n\leq  \frac 1 {12 (\ell(n))^2}$ yielding 
$\mathop{\lim}\limits_{n\to \infty} V_n=0$. Take $n$ large enough so that $ V_{3{\ell(n)}}-V_\infty<1$. Then,
\[0<-\log (V_{3{\ell(n)}}-V_\infty)\leq -\log (V_n-V_\infty) \leq -\log (V_{3({\ell(n)+1})}-V_\infty)\]
yielding
\[\frac{2 \log 3\ell(n)}{-\log (V_{3({\ell(n)+1})}-V_\infty)}\leq \frac{2\log n}{-\log (V_n-V_\infty)}\leq \frac{2 \log 3(\ell(n)+1)}{-\log (V_{6{\ell(n)}}-V_\infty)}.\]
 Notice that
\begin{align*}
\lim_{n\to \infty}\frac{2 \log 3\ell(n)}{-\log (V_{3({\ell(n)+1})}-V_\infty)}&=\lim_{n\to \infty} \frac{2 \log 3\ell(n)}{-\log \frac 1{12(\ell(n)+1)^2}}=1, \te{ and }\\
\lim_{n\to \infty} \frac{2 \log 3(\ell(n)+1)}{-\log V_{3{\ell(n)}}}&=\lim_{n\to \infty} \frac{2 \log 3(\ell(n)+1)}{-\log \frac 1{12(\ell(n))^2}}=1.
\end{align*}
 Hence, by the squeeze theorem, $\lim_{n\to \infty}  \frac{2\log n}{-\log (V_n-V_\infty)}=1$, i.e., the conditional unconstrained quantization dimension $D(P)$ of the probability measure $P$ exists and $D(P)=1$.
 Since
\begin{align*}
&\lim_{n\to \infty} n^2 (V_n-V_\infty)\geq \lim_{n\to \infty} (3\ell(n))^2 (V_{3(\ell(n)+1)}-V_\infty)=\lim_{n\to\infty}(3\ell(n))^2\frac 1{12(\ell(n)+1)^2}=\frac 34, \te{ and } \\
&\lim_{n\to \infty} n^2 (V_n-V_\infty)\leq \lim_{n\to \infty} (3(\ell(n)+1))^2 (V_{3\ell(n)}-V_\infty)=\lim_{n\to\infty}(3(\ell(n)+1))^2\frac 1{12(\ell(n))^2}=\frac 34,
\end{align*}
by the squeeze theorem, we have $\lim_{n\to \infty} n^2 (V_n-V_\infty)=\frac 34$, i.e., the $D(P)$-dimensional conditional unconstrained quantization coefficient for $P$ exists as a finite positive number and equals $\frac 34$.  Thus, the proof of the theorem is complete. 
\end{proof} 

\section{Observations} \label{secob}
Notice that the conditional optimal sets of $n$-points have two classifications: the conditional constrained optimal sets of $n$-points and the conditional unconstrained optimal sets of $n$-points; the corresponding quantization errors are, respectively, called the $n$th conditional constrained quantization error and the $n$th conditional unconstrained quantization error. 
\medskip 

In the following remarks, we give some observations. 

\begin{remark} \label{REM1} 
In constrained quantization it is known that elements in an optimal set of $n$-points are not necessarily the means, i.e., not necessarily the conditional expectations in their own Voronoi regions (see \cite{PR1}). Obviously, this fact is also true in conditional constrained quantization. In conditional unconstrained quantization, let $\ga\uu\gb$ be a conditional unconstrained optimal set of $n$-points for a probability distribution $P$ for some $n>\te{card}(\gb)$. Then, the distortion error contributed by an element $c\in \ga\setminus\gb$ is given by 
\[\int_{M(c|\ga\uu \gb)}\rho(x, c) \,dP(x),\]
which is minimum if $c=E(X : X\in M(b|\ga\uu \gb))$, i.e., in conditional unconstrained quantization, the elements in an optimal set of $n$-points which are not in the conditional set, are the means, i.e., the conditional expectations in their own Voronoi regions. 
\end{remark} 

\begin{remark}
Let $P$ be a uniform distribution on the closed interval $[0, 1]$. Let $V_n$ denote the $n$th unconstrained quantization error, and let $V_{c, n}$ denote the $n$th conditional unconstrained quantization error with respect to the conditional set $\gb:=\set{0}$ for the uniform distribution $P$. Then, by Proposition~\ref{Me030}, we know that the conditional optimal sets of $n$-points for all $n\geq 1$ are given by 
\[\Big\{\frac{2(j-1)}{2n-1} : 1\leq j\leq n\Big\} \te{ with } V_{c, n}(P)=\frac 1{3(2n-1)^2}.\]
On the other hand, the optimal sets of $n$-means are given by 
\[\Big\{\frac{2j-1}{2n} : 1\leq j\leq n\Big\} \te{ with } V_{n}(P)=\frac 1{12 n^2}.\]
Notice that for all $n\geq 1$, the conditional optimal sets of $n$-points exist with 
\[V_{c, n}=\frac 1{3(2n-1)^2}=\frac 1{3( n+n-1)^2}>\frac 1{3( n+n)^2}=\frac 1{12 n^2}=V_n.\]
But, the conditional quantization coefficients in both the scenarios exist as a finite positive number and equal to $\frac 1{12}$. The conditional quantization dimensions in both the scenarios also exist and equal to one. 
\end{remark} 

\begin{remark} Let $P$ be a Borel probability measure the support of which contains at least $n$ elements for some positive integer $n$. Then, the followings are true: $(i)$ In constrained or unconstrained quantization an optimal set containing exactly one element always exists; $(ii)$ in unconstrained quantization an optimal set of $n$-means containing exactly $n$ elements always exists; and $(iii)$ in constrained quantization, though the support contains at least $n$ elements or infinitely many elements, an optimal set containing more than one element may not exist, for example, one can take $P$ as a uniform distribution with support the closed interval $[a, b] $ for some $a, b\in \D R$ with $a<b$, and the constraint as the vertical line passing through the point $(\frac 12(a+b), 0)$. For more examples, see \cite{PR1}.  In constrained (or unconstrained) scenario, let $k$ be the largest positive integer for which an optimal set of $k$-points (or $k$-means) contains exactly $k$ elements, then there is no conditional optimal set containing $n$ elements for any $n\geq (k+1)$. 
\end{remark} 

\begin{remark}
Optimal sets of $n$-points for all positive integers can exist, but a conditional optimal set may not exist for all $n$. In this regard, in unconstrained scenario, we give the following example:
Let $P$ be a uniform distribution on the closed interval $[0, 1]$.  Let $\gb:=\set{(0, \frac 1{100})}$ be the conditional set. Let $\ga:=\set{(0, \frac 1{100})}\uu \set{(t_j, 0) : 1\leq j\leq n}$ be a conditional optimal set of $(n+1)$-points such that $t_1<t_2<\cdots<t_n$. Let the boundary of the Voronoi regions of $(0, \frac 1{100})$ and $(t_1, 0)$ intersect the support of $P$ at the point $(d, 0)$. Clearly $0\leq d<\frac  {t_1}2$.
Then, by \cite[Theorem~2.1.1]{RR}, we have $t_j=d+\frac {(2j-1)(1-d)}{2n}$ for $j=1, 2, \cdots n$.  
Hence, the distortion error is given by 
\begin{align*}
V(P;\ga)&=\te{distortion error contributed by } (0, \frac 1{100}) +\te{distortion error contributed by all } t_j\\
&=\int_0^d \rho((t, 0), (0, \frac 1{100}))\,dt+\frac{(1-d)^3}{12 n^2}\\
&=\frac{d^3}{3}+\frac{d}{10000}+\frac{(1-d)^3}{12 n^2}
\end{align*}
the minimum value of which is 
\begin{equation*} \label{eqMe0001} \frac{n \left(4 n \left(n \sqrt{10001-4 n^2}+249925\right)-10001 \sqrt{10001-4 n^2}\right)+250075}{750000 \left(1-4 n^2\right)^2}
\end{equation*}
and it occurs when $d=\frac{50-n \sqrt{10001-4 n^2}}{50-200 n^2}$. 
Notice that $d$ is a decreasing sequence of real numbers for $n=1, 2, \cdots, 50$, and becomes imaginary if $n\geq 51$. In fact, we see that $d>0$ if $1\leq n<50$, and $d=0$ if $n=50$.   
Hence, we can say that the conditional optimal sets of $n$-points exist for $1\leq n\leq 49$, and it does not exist if $n\geq 50$. But, notice that the optimal sets of $n$-points (indeed $n$-means) for all $n\geq 50$ still exist which are given by $\set{(\frac{2j-1}{2n}, 0) : 1\leq j\leq n}$ with quantization error $\frac 1{12n^2}$. 
\end{remark}

\begin{remark}
In Section~\ref{secMe4}, we have seen that for a uniform distribution defined on the equilateral triangle,  the conditional unconstrained quantization dimension is one, and the conditional unconstrained quantization coefficient is $\frac 34$. These values are the same as the unconstrained quantization dimension and the unconstrained quantization coefficient for the uniform distribution on the equilateral triangle (see \cite{RR}).  In the next section, we give some general proofs to show whether the lower and upper quantization dimensions, and the lower and upper quantization coefficients for a Borel probability measure depend on the underlying conditional set. 
\end{remark} 

\section{Some important properties} \label{prop}
In this section, first we show that in unconstrained quantization, the union of all optimal sets of \( n \)-means is dense in the support of \( P \). Second, we prove that in conditional constrained quantization, the lower and upper quantization dimensions and coefficients are independent of the choice of the conditional set provided that the conditional set is contained in the union of the family of constraints. If the conditional set is not contained in the union of the constraint family, the facts may not be true, as demonstrated by Example~\ref{exam1} and Example~\ref{exam2}.
\begin{theorem} \label{TheoM1} 
Let $P$ be a Borel probability measure on $\D R^k$. Let $\ga_n$ be an optimal set of $n$-means for $P$ for all $n\in \D N$. Then, $\mathop{\uu}\limits_{n=1}^\infty \ga_n$ is dense in the support of $P$. 
\end{theorem} 
\begin{proof}
Let $x\in \text{Supp}(P)$. Our aim is to show that for each $\epsilon>0$,  $B(x,\epsilon)\ii\mathop{\UU}\limits_{n=1}^{\infty}\ga_n\ne \emptyset.$
We prove it by contradiction. Let there exists an $\epsilon>0$ such that $B(x,\epsilon)\ii \mathop{\UU}\limits_{n=1}^{\infty}\ga_n=\emptyset.$ Then, 
\begin{align*}
	V_{n,r}(P)=\int d(x,\ga_n)^{r} dP(x)\geq \epsilon^{r} P(B(x,\epsilon)\cap \text{Supp}(P)).
	\end{align*}
We claim that $P(B(x,\epsilon)\cap \text{Supp}(P))>0$. Assume that $P(B(x,\epsilon)\cap \text{Supp}(P))=0$.  This implies that $x\notin \text{Supp}(P)$, which is a contradiction. Thus  $P(B(x,\epsilon)\cap \text{Supp}(P))>0$.  Therefore, we get 
$$\lim_{n\to \infty}V_{n,r}(P)\geq \epsilon^{r} P(B(x,\epsilon)\cap \text{Supp}(P))>0,$$
which contradicts the fact that $\lim_{n\to \infty}V_{n,r}(P)=0$. This implies that for  each $\epsilon>0$,  we have $B(x,\epsilon)\cap \mathop {\UU}\limits_{n=1}^{\infty}\ga_n\ne \emptyset.$ Thus, $\mathop{\UU}\limits_{n=1}^\infty \ga_n$ is dense in the support of $P$. This completes the proof.
\end{proof}

\begin{theorem} \label{TheoM2} 
Let $P$ be a Borel probability measure on $\D R^k$. In conditional constrained quantization, let the conditional set is contained in the union of the family of constraints. Then, in the conditional constrained quantization, the lower and upper quantization dimensions, and the lower and upper quantization coefficients for a Borel probability measure do not depend on the conditional set. 
\end{theorem}

\begin{proof} 
In conditional constrained quantization, let $\set{S_j\ci \D R^k: j\in \D N}$ be the family of constraints, and let the set $\gb$ be the conditional set with $\te{card}(\gb)=\ell$ for some $\ell\in \D N$ such that the following condition is true: 
\begin{equation*}
\gb \ci \UU_{j\in \D N} S_j.
\end{equation*}
Let $q$ be the least positive integer such that $\gb\ci \mathop{\uu}\limits_{j=1}^q S_j$.  
Let  $ V_{n,r}(P)$ and  $V_{c,n,r}(P)$ denote  the $n$th constrained  and the $n$th conditional constrained quantization  errors, respectively. 
We now prove the following claim.

\tit{Claim.  Let $n> \max \set{q, \ell}$. Then, 
$V_{n,r}(P)\leq V_{c, n,r}(P)\leq V_{n-\ell,r}(P).$}

The proof of $V_{n,r}(P)\leq V_{c, n,r}(P)$ is trivial. We now give the proof of $V_{c, n,r}(P)\leq V_{n-\ell,r}(P)$. 
Take any $\gg\ci \mathop{\uu}\limits_{j=1}^q S_j$ with $1\leq \te{card}(\gg)\leq n-\ell$. Then, for any $x\in \D R^k$, we have 
\[\min_{a\in \gg\uu \gb}d(x, a)^r \leq \min_{a\in \gg} d(x, a)^r.\]
Hence, 
\begin{align*} 
V_{c, n, r}(P)&=\inf_{\ga} \Big\{\int \mathop{\min}\limits_{a\in\ga\uu\gb} d(x, a)^r dP(x) : \ga \ci \mathop{\uu}\limits_{j=1}^q S_j, ~ 1\leq  \text{card}(\ga) \leq n-\ell \Big\}\\
&\leq \int \mathop{\min}\limits_{a\in\gg\uu\gb} d(x, a)^r dP(x)\\
&\leq \int \mathop{\min}\limits_{a\in\gg} d(x, a)^r dP(x),
\end{align*} 
which yields the fact that 
\[V_{c, n, r}(P)\leq \inf_{\gg} \Big\{\int \mathop{\min}\limits_{a\in\gg } d(x, a)^r dP(x) : \gg\ci \mathop{\uu}\limits_{j=1}^q S_j, ~ 1\leq  \text{card}(\gg) \leq n-\ell \Big\}=V_{n-\ell, r}(P).\]
Thus, the proof of the claim is complete. 

By the claim, using the squeeze theorem, we have $V_{\infty, r}(P)=V_{c, \infty, r}(P)$. Hence, 
\begin{align}\label{eq8}
	V_{n,r}(P)-V_{\infty, r}(P)\leq V_{c, n,r}(P)-V_{c, \infty, r}(P) \leq V_{n-\ell,r}(P)-V_{\infty, r}(P).
\end{align}
By the inequalities in \eqref{eq8}, and applying the squeeze theorem, we conclude that the lower and upper quantization coefficients are identical in both the constrained and the conditionally constrained cases. 
 Choose $n\in \mathbb{N}$ large enough such that $V_{n-\ell,r}(P)-V_{\infty, r}(P)<1$. Then, by the inequalities in \eqref{eq8}, we get  
 $$\frac{r\log n}{-\log(V_{n,r}(P)-V_{\infty, r}(P))}\leq \frac{r\log n}{-\log(V_{c, n,r}(P)-V_{c, \infty, r}(P))}\leq \frac{r\log n}{-\log(V_{n-\ell,r}(P)-V_{\infty, r}(P))}.$$
Hence, by the squeeze theorem, it follows that the lower and upper quantization dimensions are equal in both the constrained and the conditionally constrained settings. This completes the proof of the theorem.
\end{proof} 

As unconstrained quantization is a special case of constrained quantization, Theorem~\ref{TheoM2} leads to the following corollary. 
\begin{cor}
Let $P$ be a Borel probability measure on $\D R^k$. In conditional unconstrained quantization, the lower and upper quantization dimensions, and the lower and upper quantization coefficients for a Borel probability measure do not depend on the conditional set. 
\end{cor}

\begin{remark} \label{remM345} 
In the conditional constrained quantization, if the conditional set is not contained in the union of the family of constraints, then the lower and upper quantization dimensions, and the lower and upper quantization coefficients for a Borel probability measure may depend on the conditional set, indeed they may not exist depending on the conditional set,  as illustrated in the following two examples.
\end{remark}
The following example shows that the conditional constrained quantization dimension coincides with the constrained quantization dimension, but the corresponding quantization coefficients differ. 
\begin{exam} \label{exam1} 
 Let $P$ be a uniform probability measure on the closed interval $[0, 1]$. In a constrained quantization for $P$, let $S:=\set{(x, \frac 14 x+\frac 1 4) : x\in \D R}$ be a single constraint. If there is no conditional set, then comparing with Theorem~\ref{sec2Theorem1}, we see that 
$a=0, \, b=1, \, m=\frac 1 4$, and $c=\frac 14$. Hence, with respect to the squared Euclidean distance, the set $\{(a_i, \frac 14 a_i+\frac 14): 1\leq i \leq n\}$ forms a constrained optimal set of $n$-points, where $a_i=-\frac{-16 i+n+8}{17 n}$ for $1\leq i\leq n$ for any $n\in \D N$. If $n\geq 3$, then by Theorem~\ref{sec2Theorem1}, the constrained quantization error is given by 
\[V_n=\frac{3 n^3+18 n^2-10 n-12}{51 n^3} \te{ implying } V_\infty=\frac{1}{17}.\]
Hence, the constrained quantization dimension and the constrained quantization coefficient are obtained as 
\begin{equation} \label{Megha12} 
\lim_{n\to \infty}\frac{2\log n}{-\log (V_n-V_\infty)}=2 \te{ and } \lim_{n\to\infty} n(V_n-V_\infty)=\frac{6}{17}.
\end{equation} 
We now calculate the conditional constrained optimal sets of $n$-points and the $n$th conditional constrained quantization errors taking the conditional set as $\gb:=\set{(0, 0)}$ and the set $S=\set{(x, \frac 14 x+\frac 1 4) : x\in \D R}$ as a single constraint. For any $n\geq 1$, if $\ga_{n+1}$ is a conditional constrained optimal set of $(n+1)$-points, then with respect to $\ga_{n+1}$ we see that the Voronoi region of $(0, 0)$ has always positive probability. Let the boundary of the Voronoi region of $(0,0)$ intersect the support of $P$ at the element $(d, 0)$ for some $d>0$. Thus, putting $a=d, \, b=1, \, m=\frac 1 4$, and $c=\frac 14$ by Theorem~\ref{sec2Theorem1},  $\ga_{n+1}$ can be represented by
\[\ga_{n+1}= \{(0, 0)\}\UU \Big\{(a_i, \frac 14 a_i+\frac 14) : 1\leq i\leq n\Big\} \te{ where } a_i=-\frac{8 d (2 i-2 n-1)-16 i+n+8}{17 n}.\]
Notice that in the optimal set $\ga_{n+1}$, the adjacent element to $(0,0)$ is the element $(a_1, \frac 14 a_1+\frac 14)$, where $a_1=\frac{16 d n-8 d-n+8}{17 n}$, and the point $(d,0)$ is on the common boundary of their Voronoi regions. Hence, solving the canonical equation: 
\[\rho((0, 0), (d, 0))-\rho((d, 0), (a_1, \frac 14 a_1+\frac 14))=0,\]
we have $d=\frac{n^2+n\sqrt{17 n^2+52} -4}{16 n^2-4}$. Again, recall Theorem~\ref{sec2Theorem1}. If $V_{n+1}$ is the $(n+1)$th conditional constrained quantization error, where $n\geq 3$, then 
\begin{align*}
V_{n+1}&=\te{disotortion error due to the element } (0, 0) \\
&\qquad \qquad + \te{disotortion error due to the elements } (a_i, \frac 14 a_i+\frac 14) \te{ for } 1\leq i\leq n\\
&=\frac{d^3}{3}+\frac{d^2 \left(3 n^3-12 n^2+26 n-12\right)+2 d \left(3 n^3-3 n^2-8 n+12\right)+3 n^3+18 n^2-10 n-12}{51 n^3}. 
\end{align*} 
Next substituting $d=\frac{n^2+\sqrt{17 n^2+52} n-4}{16 n^2-4}$ in the above expression for $V_{n+1}$, we see that $V_{n+1}$ is a function of $n$, and then obtain 
\begin{align*}
& V_{\infty}=\lim_{n\to \infty} V_{n+1} =\frac{29 \sqrt{17}+229}{3072}.
\end{align*}
Moreover, we obtain
\begin{align} \label{Megha13}  
\lim_{n\to \infty}\frac{2\log n}{-\log (V_n-V_\infty)}=2 \te{ and } \lim_{n\to\infty} n(V_n-V_\infty)=\frac{1}{544} \left(179-5 \sqrt{17}\right).
\end{align} 
Comparing \eqref{Megha12} and \eqref{Megha13}, in this example, we see that the conditional constrained quantization dimension is same as the constrained quantization dimension, but the conditional constrained quantization coefficient and the constrained quantization coefficient are different.
 \end{exam} 
The following example demonstrates that neither the conditional constrained quantization dimension nor the conditional constrained quantization coefficient exist. However, both exist when the conditional set is removed.
\begin{exam} \label{exam2} 
 Let $P$ be a uniform probability measure on the closed interval $[0, 1]$. In a constrained quantization for $P$, let $S:=\set{(x, x+4) : x\in \D R}$ be a single constraint. If there is no conditional set, then comparing with Theorem~\ref{sec2Theorem1}, we see that 
$a=0, \, b=1, \, m=1$, and $c=4$. Hence, with respect to the squared Euclidean distance, the set $\{(a_i,   a_i+4): 1\leq i \leq n\}$ forms a constrained optimal set of $n$-points, where $a_i=\frac{2 i-1}{4 n}-2$ for $1\leq i\leq n$ for any $n\in \D N$. If $n\geq 3$, then by Theorem~\ref{sec2Theorem1}, the constrained quantization error is given by 
\[V_n=\frac{192 n^3+252 n^2-271 n-48}{24 n^3} \te{ implying } V_\infty=8.\]
Hence, the constrained quantization dimension and the constrained quantization coefficient are obtained as 
\begin{equation}
\lim_{n\to \infty}\frac{2\log n}{-\log (V_n-V_\infty)}=2 \te{ and } \lim_{n\to\infty} n(V_n-V_\infty)=\frac{21}{2}.
\end{equation} 
We now investigate the conditional constrained optimal sets of $n$-points and the $n$th conditional constrained quantization errors taking the conditional set as $\gb:=\set{(0, 0)}$ and the set $S=\set{(x, x+4) : x\in \D R}$ as a single constraint. Notice that for any $x\in \te{supp}(P)$, we have
\[d(x, (0, 0))<\min \set{d(x, a) : a \in S}.\]
Hence, the conditional constrained optimal sets of 
$n$-points, containing exactly $n$ elements, do not exist for any $n\geq 2$. In other words, a decreasing sequence $V_n$ of conditional constrained quantization errors can not be obtained in this scenario. Consequently, neither the conditional constrained quantization dimension nor the conditional constrained quantization coefficient exists, although both exist when the conditional set is removed. 
 \end{exam}

\end{document}